\documentclass[11pt, reqno]{amsart}
\usepackage{amsmath,amssymb,amscd,mathrsfs,amscd}
\usepackage{graphics,verbatim}
\usepackage{amsmath,amstext,amsbsy,amssymb,amscd}
\usepackage{amsmath}
\usepackage{amsxtra}
\usepackage{amscd}
\usepackage{amsthm}
\usepackage{amsfonts}
\usepackage{graphicx}%
\usepackage{amssymb}
\usepackage{eucal}
\usepackage{color}
\usepackage{multirow}
\usepackage[enableskew,vcentermath]{youngtab}

\newtheorem{thm}{Theorem}[section]
\theoremstyle{plain}

\newtheorem{lem}[thm]{Lemma}
\newtheorem{prop}[thm]{Proposition}
\newtheorem{prp}[thm]{Proposition}
\newtheorem{cor}[thm]{Corollary}

\theoremstyle{definition}
\newtheorem{dfn}[thm]{Definition}
\newtheorem{example}[thm]{Example}

\theoremstyle{remark}
\newtheorem{rem}[thm]{Remark}
\definecolor{A}{rgb}{.75,1,.75}

\numberwithin{equation}{section}

\newcommand{\A}{\mathcal A}

\newcommand{\Z}{\mathbb Z}

\newcommand{\hf}{\frac12}

\newcommand{\ep}{\epsilon}

\newcommand{\ov}{\overline}
\newcommand{\osp}{\mathfrak{osp}}

\newcommand{\ffpr}{{}'{\bf f}}
\newcommand{\g}{\mathfrak g}
\newcommand{\und}[1]{{#1}^{+}}
\newcommand{\ev}[1]{{#1}_{\bar{0}}}
\newcommand{\od}[1]{{#1}_{\bar{1}}}

{\vskip-\lastskip\medskip
  \noindent
  {\em #1.}\enspace
  }%
{\qed\par\medskip
  }

  \newcommand{\af}{\alpha}

\newcommand{\gm}{\gamma}
\newcommand{\dt}{\delta}

\newcommand{\te}{\theta}
\newcommand{\ld}{\lambda}

\newcommand{\Gm}{\Gamma}

\newcommand{\Ld}{\Lambda}

\newcommand{\del}{\partial}

\newcommand{\dpm}{\del^{\pm}}
\newcommand{\epm}{\mathbf{e}^\pm}
\newcommand{\bfe}{\mathbf{e}}

\newcommand{\Q}{{\mathbb{Q}}}
\newcommand{\bbk}{{\mathbb{K}}}

\newcommand{\curlyQ}{{\Q(q)^{\pi}}}

\newcommand{\curlyP}{{\mathcal{P}}}
\newcommand{\calQ}{{\mathsf{Q}}}

\newcommand{\Sym}{{\mathcal{S}^-}}
\newcommand{\Symz}{{\mathcal{S}^-_0}}

\newcommand{\Hm}{\mathcal{H}^-}
\newcommand{\NH}{\mathcal{NH}}
\newcommand{\NC}{\mathcal{NC}}
\newcommand{\NHm}{\mathcal{NH}^-}
\newcommand{\NCm}{\mathcal{NC}^{-}}

\newcommand{\fqp}{{\mathbf{f}^\pi}}
\newcommand{\fq}{{\mathbf{f}^+}}
\newcommand{\fqm}{{\mathbf{f}^-}}
\newcommand{\fap}{{}_\A{\mathbf{f}^\pi}}
\newcommand{\fa}{{}_\A{\mathbf{f}^+}}
\renewcommand{\fam}{{}_\A{\mathbf{f}^-}}

\newcommand{\Sl}{{\mathfrak{sl}}}

\newcommand{\Hom}{{\operatorname{Hom}}}
\newcommand{\Homm}{{\operatorname{Hom}}^-}

\newcommand{\HOMm}{{\operatorname{HOM}}^-}
\newcommand{\End}{{\operatorname{End}}}

\newcommand{\Rad}{{\operatorname{Rad}}}

\newcommand{\Mod}{{\operatorname{Mod}}}
\newcommand{\Modm}{{\operatorname{Mod}^-}}
\newcommand{\Rep}{{\operatorname{Rep}}}
\newcommand{\Proj}{{\operatorname{Proj}\,}}
\newcommand{\Repm}{{\operatorname{Rep}^-}}
\newcommand{\Projm}{{\operatorname{Proj}^-}}

\newcommand{\Res}{{\operatorname{Res}}}
\newcommand{\Ind}{{\operatorname{Ind}}}

\newcommand{\height}{{\operatorname{ht}}}

\newcommand{\ch}{{\operatorname{ch}}}
\newcommand{\Qch}{{\operatorname{Ch}}}

\newcommand{\zero}{{\bar{0}}}
\newcommand{\one}{{\bar{1}}}

\newcommand{\ui}{{\underline{i}}}
\newcommand{\uj}{{\underline{j}}}
\newcommand{\uk}{{\underline{k}}}
\newcommand{\ul}{{\underline{l}}}

\newcommand{\lan}{{\langle}}
\newcommand{\ran}{{\rangle}}

\newcommand{\andeqn}{\,\,\,\,\,\, {\mbox{and}} \,\,\,\,\,\,}

\usepackage[all,knot]{xy}
\xyoption{arc}

\begin{document}

\title
[Categorification of superalgebras]{Categorification of quantum
Kac-Moody superalgebras}
\author[David Hill and Weiqiang Wang]{David Hill and Weiqiang Wang}
\address{Department of Mathematics, University of Virginia, Charlottesville, VA 22904}
\email{deh4n@virginia.edu (D. Hill), \quad ww9c@virginia.edu (W.
Wang)}


\begin{abstract}
We introduce a non-degenerate bilinear form and use it to provide a
new characterization of quantum Kac-Moody superalgebras with no
isotropic odd simple roots. We show that the spin quiver Hecke
algebras introduced by Kang-Kashiwara-Tsuchioka provide a
categorification of half the quantum Kac-Moody superalgebras, using
the recent work of Ellis-Khovanov-Lauda. A new idea here is that a
super sign is categorified as spin (i.e., the parity-shift functor).

\end{abstract} \maketitle

\setcounter{tocdepth}{1}
\tableofcontents

\section{Introduction}

\subsection{Background}

In recent years, quiver Hecke algebras (or, KLR algebras) were
introduced independently by Khovanov-Lauda and Rouquier
\cite{KL1,KL2,Ro1}. These algebras are fundamental in the construction of
Khovanov-Lauda-Rouquier 2-categories--categorical analogues
of quantum Kac-Moody algebras (also see the earlier work \cite{CR}).
The KLR 2-categories and quiver Hecke algebras have implications in
modular representation theory of symmetric groups and Hecke
algebras, low dimensional topology, algebraic geometry, and other areas
\cite{BK2, BKW, BS, CKL, HM, HS, KK, LV, VV, Ro2, Web}; see also the
survey articles \cite{Kl2, Kh1} for more references.

Several years ago, partly motivated by Nazarov's construction of
affine Hecke-Clifford algebras \cite{Naz}, the second author
\cite{Wa} introduced {\em spin Hecke algebras} and studied them in a
series of papers with Khongsap starting in \cite{KW}. The spin Hecke
algebras are associated to spin Weyl groups, and they afford many
variations (e.g., affine, double affine, degenerate, nil, etc.). A
distinct new feature in \cite{Wa, KW} is the
appearance of the skew-polynomial algebra as a subalgebra of the
spin Hecke algebra, in contrast to the polynomial algebra for the
stardard affine/double Hecke algebras. The spin Hecke algebras are
naturally superalgebras, and are Morita super-equivalent to
Hecke-Clifford algebras (though, we prefer to suppress the prefix
``super" for such associative algebras in contrast to Lie
superalgebras). A straightforward modification of the  spin Hecke
algebra is the spin nilHecke algebra, which has recently been
rediscovered and studied in depth in  \cite{EKL} for the affine type
$A$ case.

Almost as recently, Kang-Kashiwara-Tsuchioka \cite{KKT}  utilized
the spin nilHecke algebra to generalize the KLR construction to
several new families of algebras, including the \emph{spin quiver
Hecke algebras} (called ``quiver Hecke superalgebras'' in \emph{loc.
cit.}) and quiver Hecke-Clifford algebras, starting from a
generalized Cartan matrix (GCM) $A$ parametrized by an index set
$I=I_\zero \cup I_\one$ subject to some natural conditions (see
\S\ref{subsec:GCM}). Roughly speaking, to each $i\in I_\zero$ is
attached the usual nilHecke algebra and to each $i\in I_\one$ is
attached the spin nilHecke algebra; when $I_\one =\emptyset$, the
KKT construction reduces to the original KLR construction. It is
suggested in \cite{KKT} that these new algebras can be used to
categorify the quantum Kac-Moody algebra associated to the GCM $A$
with the $\Z_2$-parity forgotten (denoted by $\und A$ in this
paper). Their expectation was partly motivated by \cite{BK} where
affine Kac-Moody algebra of type $A_{2\ell}^{(2)}$ arises (also cf.
\cite{Ts}).

\subsection{What to categorify?}

It has been known much earlier \cite{Kac} that a Kac-Moody
superalgebra can be associated to a GCM $A$ exactly as specified in
\S\ref{subsec:GCM}. This class of Kac-Moody superalgebras is
distinguished among the Lie superalgebras in the following sense:
the odd simple roots are all non-isotropic, the notion of integrable
modules is defined as usual, and the super Weyl-Kac character
formula for integrable modules holds. Note however that the only
finite-dimensional simple Lie superalgebras in this class, which are
not Lie algebras,  are $\mathfrak{osp}(1|2n)$. This class of Lie
superalgebras and beyond have been quantized in \cite{Ya, BKM}, and
a generalization of Lusztig's theorem \cite{Lu1} on deforming
integrable modules was obtained in \cite{BKM}.

We propose a connection between the spin quiver Hecke algebras and
quantum Kac-Moody superalgebras, despite the major difficulty caused
by additional signs appearing in the superalgebras. These signs show
up in the super quantum integers \eqref{eq:pi=-1}, which may
specialize to zero if one naively takes $q\mapsto 1$, and also
appear in the super quantum Serre relations \eqref{eq:superserre}.
Another major conceptual obstacle is that no canonical basis \`a la
Lusztig-Kashiwara has ever been constructed (even conjecturally) for
superalgebras,  in spite of various works generalizing the
corresponding crystal basis theory.

\subsection{The main results}

In this paper, we introduce a twisted bialgebra $\fap$, called the {\em
(quantum) covering Kac-Moody algebra}, with an additional parameter
$\pi$ satisfiying $\pi^2=1$. This algebra is associated to the GCM
$A$ with an extra natural condition (C6) in \S\ref{subsec:C6} which
is assumed throughout this paper unless otherwise specified.

The assumption (C6) plays an essential role in the paper. First, we
introduce an apparently novel bar involution on $\fap$, which is the
identity on the Chevalley generators: $\bar{\te_i}=\te_i$ ($i\in
I$), as usual, but $\bar{q}=\pi q^{-1}.$ Assumption (C6) guarantees
that the quantum integers in $\fap$, and therefore the divided
powers, $\te_i^{(a)}$ ($i\in I$), are bar invariant. Additionally,
following Rouquier \cite{Ro1}, we define a family of (skew-)
polynomials $\mathsf{Q}=(\mathsf{Q}_{ij}(u,v))_{i,j\in I}$ from a
\emph{quiver with compatible automorphism} and show  in
Lemma~\ref{L:Qconditions} that assumption (C6) implies these
polynomials satisfy the necessary conditions to construct an
associated spin quiver Hecke algebra, \cite[(3.1)]{KKT}.

The two specializations, $\pi\mapsto 1$ and $\pi\mapsto-1$, of
$\fap$ become half of the Kac-Moody algebra associated to $A^+$ and
half of the Kac-Moody superalgebra associated to $A$, respectively.
The main result of this paper is that the spin quiver Hecke algebras
defined by the family of polynomials, $\mathsf{Q}$, naturally
categorify the algebra $\fap$ and, consequently, we obtain a
categorification of halves of the corresponding Kac-Moody bialgebra
and bisuperalgebra simultaneously.

Our key new idea in this paper is to distinguish two types of signs
occuring in (quantum) Kac-Moody superalgebras. Signs common to Lie
algebras and superalgebras, or ordinary signs, are denoted by $(-1)$
as usual, while  the signs arising from exchanges of odd elements
are replaced by the parameter $\pi$.  The ordinary signs are
categorified using complexes as usual following \cite{KL1, KL2,
Ro1}, while $\pi$ is the shadow of a parity-shift functor (called
\emph{spin}). Just as the parameter $q$ is categorified by an
integer grading shift, $\pi$ is categorified by a spin. Forgetting
the spin corresponds to the specialization $\pi\mapsto1$, and in
this way one ends up with the usual quantum Kac-Moody algebras (as
suggested in \cite{KKT}).

We introduce a bilinear form on the  algebra $\fap$ (which
specializes to a form on the superalgebra $\fam$) and establish its
non-degeneracy, following \cite{Lu}. With this in place, the
necessary categorical constructions can be obtained within the
framework of \cite{KL1, KL2} using the spin quiver Hecke algebras.
In the process we find the detailed structures of spin nilHecke
algebras worked out in \cite{EKL} handy to use.  As a consequence of
our categorification, we prove a conjecture of \cite{KKT} that all
simple modules of spin quiver Hecke algebras are of type $\texttt M$
(that is, they remain simple with the $\Z_2$-grading forgotten).

In the simplest case when $I =I_\one$ consists of an (odd)
singleton, the spin quiver Hecke algebra reduces to the spin
nilHecke algebra of the second author. In this case, our assertion
is that the spin nilHecke algebras categorify half of the quantum
$\mathfrak{osp}(1|2)$ (which is new) as well as half of the quantum
$\mathfrak{sl}(2)$ (which was already proved in \cite{EKL}; see also
\cite{KKT}).

\subsection{Future work}

The  ideas of this paper are expected to have several ramifications.
The results here can be rephrased in terms of $2$-Kac-Moody
superalgebras in the sense of \cite{Ro1, Ro2}. One can also imitate
\cite{Web} to formulate a (conjectural) categorification of tensor
products of integrable modules of quantum Kac-Moody superalgebras.
Following the algebraic construction of \cite{KK} it should allow
one to show that the cyclotomic spin quiver Hecke algebras
categorify the integrable modules of the quantum Kac-Moody
(super)algebras. The idea here can also be combined with \cite{KOP}
to categorify the more general quantum Borcherds  superalgebras and
their integrable modules studied in \cite{BKM}.

Another main point of this paper is the introduction for the first
time of a bar-involution on quantum superalgebras such that
$$
\bar{q}=- q^{-1},
$$
and the assumption (C6) is again  perfect for this purpose. A
remarkable property of this bar-involution is its compatibility with
the categorification. The canonical basis for the modified or
idempotented quantum $\mathfrak{osp}(1|2)$ is constructed in
\cite{CW}. In a work in preparation joint with Sean Clark, we are
undertaking a construction of the canonical bases for quantum
Kac-Moody superalgebras. It will be interesting to compare these
canonical bases with those coming from categorification.

We hope that our work helps to clarify the right framework for
categorifying the odd Khovanov homology (cf. \cite{EKL} and
references therein). A new idea was suggested in \cite{Kh2} on how
to categorify a superalgebra with an isotropic odd simple root. It
is natural to expect, though remains highly non-trivial, that the
categorification of the more general Kac-Moody superalgebras will
have to combine all these ideas of categorifying the even simple
roots, the isotropic odd simple roots, and the non-isotropic odd
simple roots.

\subsection{Organization}

The layout of this paper is as follows.  In Section \ref{S:Root
Data} we collect the necessary Lie theoretic data. In Section
\ref{S:QKM} we define the covering Kac-Moody algebra, realize it as
the quotient of a free algebra by the radical of a bilinear form,
and define a new bar involution. In Section \ref{S:sqHa} we recall
the definition of the spin quiver Hecke algeba and describe some of
its basic properties. In Section \ref{sec:Representations} we
introduce the category of finitely generated graded projective
modules over the spin quiver Hecke algebra that will categorify the
covering Kac-Moody algebra, and then establish the categorical Serre
relations. From these results, we deduce in Section~
\ref{S:Categorification} the categorification of the covering
Kac-Moody algebra.

\bigskip

{\bf Conventions.} A module over a superalgebra  $R$ in this paper
is understood as a \emph{left} module with $\Z_2$-grading compatible
with that of $R$.

\bigskip

{\bf Acknowledgement.} We are indebted to the authors of \cite{KKT,
EKL}, as this paper grew out of our readings of their papers and
relies heavily on their results. We thank Sean Clark for stimulating
discussions and collaborations. The second author is grateful to
Jinkui Wan for her collaboration in 2009 when we were motivated by
\cite{Naz, BK, Wa, KW} to make an unsuccessful attempt to construct
the spin/Clifford generalization of the KLR algebras. The research
of the second author is partially supported by NSF grant
DMS-1101268.

After the completion of this paper, Kang-Kashiwara-Oh \cite{KKO} used \cite{KKT} to categorify the Kac-Moody algebra $\g=\g(A^+)$ and its integrable representations following \cite{KK}, but the connection with Kac-Moody superalgebras was not pursued.

\section{Root data}\label{S:Root Data}

\subsection{Generalized Cartan matrices} \label{subsec:GCM}

Let $I=I_\zero \cup I_\one$ be a $\Z_2$-graded finite set of size
$\ell$. Let $A=(a_{ij})_{i,j\in I}$ be a generalized Cartan matrix
(GCM) such that
\begin{enumerate}
\item[(C1)] $a_{ii} =2$, for all $i\in I$;

\item[(C2)] $a_{ij} \in \Z_{\leq0}$, for $i\neq j \in I$;

\item[(C3)] $a_{ij}=0$ if and only if $a_{ji}=0$;

\item[(C4)] $a_{ij} \in 2\Z$, for all $i\in I_\one$ and all $j\in I$;

\item[(C5)]
there exists an invertible matrix $D =\text{diag}(s_1,\ldots,
s_r)$ with $DA$ symmetric.
\end{enumerate}
We further assume $s_i \in \Z_{>0}$ and $\gcd
(s_1,\ldots, s_r)=1$. Introduce the parity function $p(i)=0$ for
$i\in I_\zero$ and $p(i)=1$ for $i\in I_\one$.

Condition (C4) was imposed first in \cite{Kac} so that the corresponding Kac-Moody superalgebras
 possess similar favorable properties as $\mathfrak{osp}(1|2n)$, i.e., the odd simple roots are all non-isotropic
and the Weyl-Kac character formula for integrable modules expressed in terms of the Weyl group holds.

\subsection{The assumption (C6)}   \label{subsec:C6}

We will impose an additional condition on a GCM $A$ introduced in
\S\ref{subsec:GCM} for the Kac-Moody superalgebras considered in
this paper:

\vspace{.2cm}
 (C6) $I_\one\neq\emptyset$, and the integer $s_i$ is odd
if and only if $i$ is odd (i.e., $i\in I_\one$). \vspace{.2cm}

The case $I_\one=\emptyset$ has been well studied, and we
have nothing new to add.

There is precisely one Kac-Moody superalgebra of finite type satisfying (C1)--(C6). Namely, $\mathfrak{osp}(1|2n)$ (or, $B(0,n)$):
$$\xy
(-25,0)*{\bullet};(-15,0)*{\circ}**\dir{=};(-5,0)*{\circ}**\dir{-};
(0,0)*{\cdots};(-20,0)*{<};
(5,0)*{\circ};(15,0)*{\circ}**\dir{-};(25,0)*{\circ}**\dir{-};
\endxy$$
In Table 1 below, we list the affine Dynkin diagrams
satisfying the parity assumption (C6). The nodes labeled by
$I_\zero$ are drawn as hollow circles $\circ$; the nodes labeled
by $I_\one$ are drawn as solid dots $\bullet$. A complete list of affine Lie superalgebras and Dynkin diagrams can be found in \cite{vdL}
and we observe that there is exactly one family of affine superalgebras excluded by (C6).

\begin{table}[ht]\label{T:Dynkin}
 \caption{Affine Dynkin diagrams satisfying (C1)-(C6)}
\begin{center}
\begin{tabular}{|c|c||c|c|}\hline
$B^{(1)}(0,n)$&
$$ \xy
(-25,0)*{\bullet};(-15,0)*{\circ}**\dir{=};(-5,0)*{\circ}**\dir{-};
(0,0)*{\cdots};(-20,0)*{<};
(5,0)*{\circ};(15,0)*{\circ}**\dir{-};(25,0)*{\circ}**\dir{=};(20,0)*{<};
\endxy$$
&
$B^{(1)}(0,1)$&
$$\xy
(-5,0)*{\bullet};(5,0)*{\circ}**\dir2{=};(0,0)*{<};(-5,1)*{};(5,1)*{}**\dir{-};
(-5,-1)*{};(5,-1)*{}**\dir{-};
\endxy$$
\\\hline
$A^{(2)}(0,2n-1)$&
$$\xy
(-25,0)*{\bullet};(-15,0)*{\circ}**\dir{=};(-5,0)*{\circ}**\dir{-};
(0,0)*{\cdots};(-20,0)*{<};
(5,0)*{\circ};(15,0)*{\circ}**\dir{-};(22,5)*{\circ}**\dir{-};(15,0)*{\circ};(22,-5)*{\circ}**\dir{-};
\endxy$$
&
$A^{(2)}(0,3)$&
$$\xy
(-10,0)*{\circ};(0,0)*{\bullet}**\dir{=};(10,0)*{\circ}**\dir{=};
(-5,0)*{>};(5,0)*{<};
\endxy$$
\\\hline
$C^{(2)}(n+1)$&
$$ \xy
(-25,0)*{\bullet};(-15,0)*{\circ}**\dir{=};(-5,0)*{\circ}**\dir{-};
(0,0)*{\cdots};(-20,0)*{<};
(5,0)*{\circ};(15,0)*{\circ}**\dir{-};(25,0)*{\bullet}**\dir{=};(20,0)*{>};
\endxy$$
&
$C^{(2)}(2)$&
$$\xy
(-5,0)*{\bullet};(5,0)*{\bullet}**\dir{=};(-1,0)*{<};(1,0)*{>};
\endxy$$
\\\hline
\end{tabular}
\end{center}
\end{table}

\subsection{Quivers with compatible automorphism}\label{subsec:Qmatrix}

Let $\bbk$ be a field, $\mathrm{char}\bbk\neq2$. We continue to work
under the assumptions of $\S$\ref{subsec:GCM} and
$\S$\ref{subsec:C6} throughout the paper.

Let $\widetilde{\Gm}$ be a graph without loops. We construct a Dynkin diagram
$\Gm$ by giving $\widetilde{\Gm}$ the structure of a \emph{graph
with compatible automorphism} in the sense of \cite[$\S12,14$]{Lu}.
To define the quiver Hecke algebra, we will use the notion of a
\emph{quiver with compatible automorphism} as described in
\cite[$\S3.2.4$]{Ro1}.

Let $\widetilde{I}$ be the labelling set for $\widetilde{\Gm}$, and
$\widetilde{H}$ be the (multi)set of edges. An automorphism
$a:\widetilde{\Gm}\to\widetilde{\Gm}$ is said to be
\emph{compatible} with $\widetilde{\Gm}$ if, whenever $(i,j)\in
\widetilde{H}$ is an edge, $i$ is not in the orbit of $j$ under $a$
(so the quotient graph has no loops).

Fix a compatible automorphism $a:\widetilde{\Gm}\to\widetilde{\Gm}$,
and set $I$ to be a set of representatives of the obits of
$\widetilde{I}$ under $a$ and let $\Gm=\widetilde{\Gm}/a$ be the
corresponding diagram with nodes labeled by $I$. For each $i\in I$,
let $\af_i\in\widetilde{I}/a$ be the corresponding orbit. For $i,
j\in I$ with $i\neq j$, let
\begin{align}
\label{E:SymmetrizingConst}s_i&=|\af_i|,
 \\
\label{E:afiafi}(\af_i,\af_i)&=2s_i,
 \\
\label{E:afiafj}
(\af_i,\af_j)&=-|\{(i',j')\in\tilde{H}\,|\,i'\in\af_i,j'\in\af_j\}|.
\end{align}
For all $i,j\in I$, let $a_{ij}=(\af_i,\af_j)/s_i$. Then, by
\cite[Proposition~ 14.1.2]{Lu} $A=(a_{ij})_{i,j\in I}$ is a GCM and every GCM arises in this way. Moreover, the symmetrizing constants
$s_i$ are nonnegative integers by definition, and we may assume that $\gcd\{ s_i | i\in I \}=1$ (otherwise, let $\ell=\mathrm{lcm}\{s_i\mid i\in I\}/\gcd\{s_i\mid i\in I\}$, and repeat the construction above with $\widetilde{\Gm}/a^\ell$ instead). Define a $\Z_2$-grading on $I$ by setting
$I_\zero=\{i\in I\mid |\af_i|\in2\Z\}$ and $I_\one=I\backslash I_\zero$. Then, (C6) is satisfied. Among the diagrams obtained from this construction, we will work only with those satisfying (C4).

Then, $A$ is a GCM as in $\S$\ref{subsec:GCM}, and $\Gm$ is its
Dynkin diagram. We additionally have the data:
\begin{align}
\label{E:SimpleRoots}\mbox{Simple roots: }&  \{\af_i|i\in I\}\\
\label{E:RootLattice}\mbox{Root lattice: }
 & Q=\bigoplus_{i\in I}\Z\af_i,\;Q^+=\bigoplus_{i\in I}\Z_{\geq0}\af_i;\\
\label{E:RootPairing}\mbox{Bilinear pairing:  }& (\cdot,\cdot):Q\times
Q\longrightarrow \Z.
\end{align}

Assume further that $\widetilde{\Gm}$ is a quiver. That is, we have
a pair of maps $s:\widetilde{H}\to\widetilde{I}$ and
$t:\widetilde{H}\to\widetilde{I}$ (the source and the target). We
assume that $a$ is compatible with the quiver structure in the
sense that it is equivariant with respect to the source and target
maps: $s(a(h))=a(s(h))$ and $t(a(h))=a(t(h))$ for all
$h\in\widetilde{H}$.

Set
\begin{align}
d_{ij}=\bigg|\{h\in\widetilde{H}\,|\,s(h)\in\af_i\mbox{ and }t(h)\in\af_j\}/a\bigg|
\end{align}
and let $m(i,j)=\mathrm{lcm}\{(\af_i,\af_i),(\af_j,\af_j)\}
 =2s_is_j/\gcd(s_i,s_j)$. As noted in \cite{Ro1},
\begin{align}\label{E:dij+dji}
d_{ij}+d_{ji}=-2(\af_i,\af_j)/m(i,j).
\end{align}

Following \cite[$\S$3.1]{KKT}, for $i,j\in I$, define a ring
$$
\bbk_{ij}\{u,v\}=\bbk\lan u,v \ran/\lan uv-(-1)^{p(i)p(j)}vu\ran.
$$
The data above defines a matrix $\calQ=(\calQ_{ij}(u,v))_{i,j\in
I}$, see \cite{Ro1}. Each $\calQ_{ij}(u,v)\in\bbk_{ij}\{u,v\}$, and
the (skew-)polynomial entries in $\calQ$  are defined by
$\calQ_{ii}(u,v)=0$, and for $i\neq j$,
\begin{align}\label{E:Q polynomials}
\calQ_{ij}(u,v)=(-1)^{d_{ij}}
\big(u^{m(i,j)/(\af_i,\af_i)}-v^{m(i,j)/(\af_j,\af_j)}
\big)^{-2(\af_i,\af_j)/m(i,j)}.
\end{align}

\begin{lem}\label{L:Qconditions} For $i, j\in I$, we have
\begin{enumerate}
\item[(a)] $\calQ_{ij}(u,v)\in\bbk_{ij}\{u,v\}$,
\item[(b)] $\calQ_{ii}(u,v)=0$,
\item[(c)] $\calQ_{ij}(u,v)=\calQ_{ji}(v,u)$,
\item[(d)] $\calQ_{ij}(-u,v)=\calQ_{ij}(u,v)$ whenever $i\in I_\one$.
\end{enumerate}
In particular, the entries in $\calQ$ satisfy the properties in \cite[(3.1)]{KKT}.
\end{lem}

\begin{proof}
Properties (a), (b), and (c) are clear by definition.

To prove (d), first assume $j\in I_\zero$. Note that by assumption
(C6) $s_i$ is odd and $s_j$ is even. Therefore,
$m(i,j)/(\af_i,\af_i)=s_j/\gcd(s_i,s_j)$ is even, proving the lemma
in this case.

Next, assume $j\in I_\one$. In this case, both
$m(i,j)/(\af_i,\af_i)=s_j/\gcd(s_i,s_j)$ and
$m(i,j)/(\af_j,\af_j)=s_i/\gcd(s_i,s_j)$ are odd by assumption (C6).
However, $$-2(\af_i,\af_j)/m(i,j)=a_{ij}\gcd(s_i,s_j)/s_j$$ is even,
since $a_{ij}$ is even by (C4) and both $s_i$ and $s_j$ are odd
(again, by (C6)). Since $uv=-vu\in\bbk_{ij}\{u,v\}$,
$$(u^a-v^b)^{2c}=(u^{2a}+v^{2b})^c$$ whenever $a$ and $b$ are both
odd, the result follows.
\end{proof}

\section{Quantum superalgebras and bilinear forms}\label{S:QKM}

\subsection{Kac-Moody superalgebras}   \label{subsec:superKM}

Associated to a GCM $A$ as in $\S$\ref{subsec:GCM} and
$\S$\ref{subsec:C6} is a Kac-Moody superalgebra $\g =\g(A)$ (see
\cite{Kac}), and a quantized enveloping superalgebra (see
\cite{BKM}). Let $\fqm$ be the super analogue of Lusztig's algebra
$\mathbf{f}$, generated over $\Q(q)$ by $\{\te_i|i\in I\}$ and
subject to a signed analogue of Serre relations (cf.
\eqref{eq:superserre} below). The $\te_i$ parameterized by $i\in
I_\zero$ (respectively, $i\in I_\one$) are {\em even} (respectively,
{\em odd}). For $k \ge 0$ and $i\in I$, we shall denote by
$\te_i^{(k)} =\te_i^k/[k]_i^-!$ the divided powers, where
\begin{align}
q_i =q^{s_i}, \qquad [k]_i^- =\frac{(- 1)^{kp(i)} q_i^k
-q_i^{-k}}{(- 1)^{p(i)} q_i -q_i^{-1}}, \qquad [k]_i^-!
=\prod_{a=1}^k [a]_i^-. \label{eq:pi=-1}
\end{align}
We shall denote by $\A=\Z[q,q^{-1}]\subset\Q(q)$ and by $\fam$ the
$\A$-subalgebra of $\fqm$ generated by all divided powers
$\te_i^{(k)}$, for $k\ge 1, i\in I$. Besides the standard relation
among divided powers, the signed Serre relations in $\fam$ can be
written as:
\begin{align}
\sum_{k=0}^{1-a_{ij}} (-1)^{k+{p(k;i,j)}} \te_i^{(1-a_{ij}-k)} \te_j
\te_i^{(k)} &=0,
 \label{eq:superserre}
\end{align}
where
\begin{align}\label{E:pkij}
p(k;i,j) ={kp(i)p(j)+\hf k(k-1)p(i)}.
\end{align}

\begin{dfn}
Assume a GCM $A$ satisfies (C1)-(C6). We define a bar-involution
$\bar{\quad}:\fqm\to\fqm$ by letting
\begin{equation}  \label{eq:superbar}
\ov{q} =- q^{-1}, \qquad \ov{\te_i} =\te_i \quad (\forall i\in I).
\end{equation}
\end{dfn}
Note that all the divided powers $\te_i^{(k)}$ are bar-invariant
under the assumption (C6), and hence we have
$\bar{\quad}:\fam\to\fam$.

\begin{rem}
With respect to this apparently new bar-involution, we will develop
a theory of canonical basis in a forthcoming work with Sean Clark.
\end{rem}

\subsection{Kac-Moody algebras}\label{subsec:KM}

If we forget the parity on $I$ in \ref{subsec:GCM}, we shall
write the corresponding  GCM matrix by $\und A$.
Associated to the GCM $\und A$ is the usual Kac-Moody algebra $\und
\g =\g(\und A)$, and write $\fq$ for Luszig's algebra $\mathbf{f}$.
For $k \ge 0$ and $i\in I$, we shall abuse notation slightly
and denote by $\te_i^{(k)} =\te_i^k/[k]_i^+!$ the divided powers as
before, where now
\begin{align}
[k]_i^+ =\frac{(+ 1)^{kp_i} q_i^k
-q_i^{-k}}{(+ 1)^{p(i)} q_i -q_i^{-1}}, \qquad [k]_i^+! =\prod_{a=1}^k
[a]_i^+.
\label{eq:pi=1}
\end{align}
Let $\fa$ be the $\A$-subalgebra of $\fq$ generated by all divided
powers $\te_i^{(k)}$, for $k\ge 1$ and $i\in I$, subject to the
usual Serre relations, which we suggestively write as:
\begin{align}
\sum_{k=0}^{1-a_{ij}} (-1)^k (+1)^{p(k;i,j)} \te_i^{(1-a_{ij}-k)}
\te_j \te_i^{(k)} &=0. \label{eq:serre+}
\end{align}

\subsection{Covering Kac-Moody algebras}\label{subsec:KMcover}

Finally, we present a common framework to describe the presentations
in $\S$\ref{subsec:superKM} and $\S$\ref{subsec:KM}, and in the
process justify our seemingly inconsistent use of notation therein.
To this end, fix an indeterminant $\pi$, and for a ring $R$, we
introduce a new ring
$$
R^\pi=R[\pi]/(\pi^2-1).
$$

Denote by   $\ffpr^\pi =\ev\ffpr^\pi \oplus \od \ffpr^\pi$ the free
associative algebra over $\curlyQ$ generated by even generators
$\theta_i$ for $i\in \ev I$ and odd generators $\theta_i$ for $i\in
\od I$. We have parity $p(x)=0$ for $x\in \ev\ffpr^\pi$ and
$p(x)=1$ for $x\in \od\ffpr^\pi$. Letting the weight of $\theta_i$
be $\af_i\in Q^+$, the algebra $\ffpr^\pi$ has an induced weight
space decomposition $\ffpr^\pi =\oplus_{\nu \in Q^+} \ffpr_\nu^\pi$.
For $x\in \ffpr^\pi_\nu$, we set $|x| =\nu$.

For $a \ge t \ge 0$ and $i\in I$, we shall denote
\begin{align*}
q_i =q^{s_i}, \quad \pi_i=\pi^{p(i)},
\end{align*}
and
\begin{align}
[a]_i =\frac{\pi_i^{ a} q_i^a
-q_i^{-a}}{\pi_i q_i -q_i^{-1}}, \quad [a]_i! =\prod_{k=1}^a
[k]_i,
 \quad
\left[ \begin{matrix}a\\t\end{matrix}\right]_i
=\frac{[a]_i!}{[t]_i![a-t]_i!}.
\label{eq:pi}
\end{align}
We denote by $\theta_i^{(a)} =\theta_i^a/[a]_i!$ the divided
powers.

Define an algebra homomorphism $r:\ffpr^\pi \rightarrow \ffpr^\pi \otimes
\ffpr^\pi$ by letting
$$
r(\theta_i) =\theta_i\otimes 1 +1\otimes \theta_i, \qquad \forall
i\in I.
$$
The algebra structure on $\ffpr^{\pi} \otimes \ffpr^{\pi}$ is given by
\begin{equation} \label{eq:prod}
(x_1\otimes x_2)(x_1' \otimes x_2') =\pi^{p(x_2) p(x_1')}
q^{-(|x_2|, |x_1|)} x_1 x_1' \otimes x_2 x_2',
\end{equation}
for homogeneous $x_1,x_2, x_1',x_2' \in\ffpr^{\pi}$, where
$(|x_2|,|x_1'|)$ is defined in \eqref{E:RootPairing}. In particular,
we note that
\begin{equation} \label{eq:piq2}
(1\otimes \theta_i) \cdot (\theta_i\otimes 1) =\pi_i
q_i^{-2}(\theta_i\otimes 1) \cdot (1\otimes \theta_i).
\end{equation}

The following is a super analogue of \cite[Proposition~1.2.3]{Lu}.
Note that the $v$ in Lusztig corresponds to our $q^{-1}$. Though the
identities below look almost identical to those in {\em loc. cit.},
we give a detailed proof where some super signs show up.

\begin{prop}  \label{prop:bform}
There exists a unique bilinear form $(\cdot,\cdot)$ on $\ffpr^\pi$
with values in $\curlyQ$ such that $(1,1)=1$,
\begin{enumerate}
\item[(a)]
$(\theta_i, \theta_j) = \delta_{ij} (1-\pi_i q_i^{2})^{-1} \quad
(\forall i,j\in I),$

\item[(b)]
$(x, y'y'') = (r(x), y'\otimes y'') \quad (\forall x,y',y'' \in
\ffpr^{\pi}),$

\item[(c)]
$(xx', y'') = (x'\otimes x'', r(y'')) \quad (\forall x',x'',y'' \in
\ffpr^{\pi}),$

\item[(d)]
the induced bilinear form  $(\ffpr^{\pi} \otimes \ffpr^{\pi})
\times (\ffpr^{\pi} \otimes \ffpr^{\pi}) \rightarrow \Q(q)$ is given
by
\begin{equation} \label{eq:tensorform}
(x'\otimes x'', y' \otimes y'') := (x',y')(x'', y'').
\end{equation}
\end{enumerate}
Moreover, the bilinear form $(\cdot,\cdot)$ is symmetric.
\end{prop}

\begin{proof}
We follow \cite[1.2.3]{Lu} to define an associative algebra
structure on $\ffpr^* :=\oplus_{\nu} \ffpr^*_\nu$ by dualizing the
coproduct $r: \ffpr \rightarrow \ffpr \otimes \ffpr$. Let $\xi_i\in
\ffpr^*_i$ be defined by $\xi_i(\theta_j) =\dt_{ij}(1-\pi_i
q_i^{2})^{-1}$, for all $j\in I.$ Let $\phi: \ffpr \rightarrow
\ffpr^*$ be the unique algebra homomorphism such that
$\phi(\theta_i) =\xi_i$ for all $i$. We may identify
$\ffpr^*\otimes\ffpr^*\cong(\ffpr\otimes\ffpr)^*$ so that the
functional $(\phi\otimes\phi)(y\otimes y')$, for $y,y'\in \ffpr$, is
given by
\begin{equation}  \label{eq:phi2}
(\phi(y)\otimes\phi(y'))(x\otimes x')=(\phi(y)(x))(\phi(y')(x')),
\quad x,x'\in \ffpr.
\end{equation}
Note that the maps $\phi$ and $\phi\otimes\phi$ preserve the
$(Q\times \Z_2)$-grading.

Define $(x,y) =\phi(y)(x)$, for $x,y\in \ffpr$. Property (a) follows
directly from the definition, and (d) follows from \eqref{eq:phi2}.

Clearly $(x,y)=0$ unless (homogeneous) $x,y$ have the same weight in
$ Q$, which implies they must have the same parity. All elements involved below will be
assumed to be homogeneous.

Now, write $r(x) =\sum x_1\otimes x_2$. We have
\begin{align*}
(x,y'y'') &=\phi(y'y'')(x) =(\phi(y')\phi(y'')) (x)\\
    &=(\phi(y')\otimes \phi(y'')) (r(x))   \\
    &=\sum (\phi(y')\otimes \phi(y'')) (x_1\otimes x_2)\\
    &=\sum  \phi(y')(x_1) \phi(y'')(x_2)\\
    &=\sum \phi(y')(x_1)\phi(y'')(x_2) = (r(x), y'\otimes y'').
\end{align*}
This proves (b).

It remains to prove (c). The cases when $y''$ is $1$ or $\theta_j$
can be checked directly. Assume that (c) is known for $y''$ replaced
by $y$ or $y'$ and for any $x, x'$. We then prove that (c) holds for
$y''=yy'$. Write
\begin{align*}
r(x) =\sum x_1\otimes x_2, \quad & r(x') =\sum x_1'\otimes x_2',
  \\
r(y) =\sum y_1\otimes y_2, \quad & r(y') =\sum y_1'\otimes y_2'.
\end{align*}
Then
\begin{align*}
r(xx') &=\sum q^{(|x_2|,|x_1'|)}\pi^{p(x_2)p(x_1')} x_1x_1'\otimes
x_2x_2',
  \\
r(yy') &=\sum q^{(|y_2|,|y_1'|)}\pi^{p(y_2)p(y_1')} y_1y_1'\otimes
y_2y_2'.
\end{align*}
We have
\begin{align}
(xx',yy') &=(\phi(y)\phi(y'))(xx')=(\phi(y)\otimes\phi(y'))(r(xx'))
   \notag \\
&= \sum q^{-(|x_2|, |x_1'|)} \pi^{p(x_2)p(x_1')}
(x_1x_1',y)(x_2x_2',y')
   \notag \\
&= \sum q^{-(|x_2|, |x_1'|)} \pi^{p(x_2)p(x_1')}
(x_1 \otimes x_1',r(y))(x_2\otimes x_2',r(y'))
   \notag \\
&= \sum q^{-(|x_2|, |x_1'|)}
\pi^{p(x_2)p(x_1')} (x_1,y_1)(x_1',y_2)(x_2,y_1')(x_2',y_2').
  \label{eq:d}
\end{align}
On the other hand,
\begin{align}
(x\otimes x', r(yy'))
 &=\sum q^{-(|y_2|,|y_1'|)}\pi^{p(y_2)p(y_1')}
(x\otimes x', y_1y_1'\otimes y_2y_2')
   \notag \\
&=\sum q^{-(|y_2|,|y_1'|)}\pi^{p(y_2)p(y_1')} (x,y_1y_1')(x',y_2y_2')
   \notag \\
&=\sum q^{-(|y_2|,|y_1'|)}\pi^{p(y_2)p(y_1')}
(r(x),y_1\otimes y_1')(r(x'),y_2\otimes y_2')
   \notag \\
&=\sum q^{-(|y_2|,|y_1'|)}
\pi^{p(y_2)p(y_1')}
(x_1,y_1)(x_1',y_2)(x_2,y_1')(x_2',y_2').
  \label{eq:e}
\end{align}
For a summand to make nonzero contribution, we must have
$|x_1'|=|y_2|$ and $|x_2|=|y_1'|$ in $Q$ and, therefore, both
$p(x_1')=p(y_2)$ and $p(x_2)=p(y_1')$. It follows that the powers of
$q$ and $\pi$ in \eqref{eq:d} and \eqref{eq:e} match perfectly.
Hence the two sums in \eqref{eq:d} and \eqref{eq:e} are equal,
proving (c).
\end{proof}

We then define (half of) the {\em covering Kac-Moody algebra} to be the
quotient algebra $\ffpr^{\pi}$ by the radical as
$$
\fqp =\ffpr^{\pi} /\text{Rad} (\cdot, \cdot).
$$
The bilinear form $(\cdot,\cdot)$ on $\ffpr$ descends to a
non-degenerate bilinear form $(\cdot,\cdot)$ on $\fqp$ satisfying
the same properties as in Proposition~\ref{prop:bform}, and we also
have an induced algebra homomorphism $r: \fqp \rightarrow \fqp
\otimes \fqp$.

\subsection{Binomial identities}
For an indeterminate $v$, let
$$
[a]_v=\frac{v^a -v^{-a}}{v -v^{-1}},
 \qquad \left[
\begin{matrix}a\\t\end{matrix}\right]_v =\frac{[a]_v!}{[t]_v![a-t]_v!},
\quad \text{for } 0\le t \le a.
$$
Recall from \cite[1.3.5]{Lu} the quantum binomial formula: for any
$a\ge 0$ and $x,y$ in a $\Q(v)$-algebra such that $xy=v^2yx$,
\begin{equation}  \label{eq:binom}
(x+y)^a = \sum_{t=0}^a v^{t(a-t)}  \left[
\begin{matrix}a\\t\end{matrix}\right]_v y^t x^{a-t}.
\end{equation}

Now we convert \eqref{eq:binom} to notations in our super setting.

\begin{lem}  \label{lem:odbinom}
For any $a\ge 0$ and $x,y$ in a $\Q(q)$-algebra such that
$xy=\pi_i q_i^{-2}yx$,
\begin{equation}  \label{eq:binom-od}
(x+y)^a = \sum_{t=0}^a (\pi_iq_i)^{-t(a-t)}  \left[
\begin{matrix}a\\t\end{matrix}\right]_i y^t x^{a-t}.
\end{equation}
\end{lem}

\begin{proof}
For $i$ even, $p(i)=0$ and so \eqref{eq:binom-od} is simply
\eqref{eq:binom} with $v=q_i^{-1}$.

Now assume $i$ is odd, i.e, $p(i)=1$, and so $\pi_i=\pi$.  Fix a
square root $\sqrt{\pi}$ of $\pi$ once for all. We have $xy =v^2yx$,
if we introduce a new indeterminate $v$ by letting
\begin{equation}  \label{eq:vq}
v:=\sqrt{\pi}q_i^{-1}.
\end{equation}
Hence \eqref{eq:binom} is applicable, and we have
$$
(x+y)^a = \sum_{t=0}^a (\sqrt{\pi}q_i^{-1})^{t(a-t)}  \left[
\begin{matrix}a\\t\end{matrix}\right]_v y^t x^{a-t}.
$$
This identity can then converted to \eqref{eq:binom-od} by using the
following identities: for $a\ge t \ge 0$,
\begin{equation}  \label{eq:Zvq}
[a]_v =\sqrt{\pi}^{a-1} [a]_i, \qquad
 [a]_v! =\sqrt{\pi}^{\frac{a(a-1)}{2}} [a]_i!, \qquad
\left[\begin{matrix}a\\t\end{matrix}\right]_v = \sqrt{\pi}^{t(a-t)}
\left[
\begin{matrix}a\\t\end{matrix}\right]_i.
\end{equation}
Since $\pi=\pi^{-1}$, this proves the lemma.
\end{proof}

\subsection{Bilinear form}

We now study the divided powers in relation to the homomorphism $r$
and the bilinear form $(\cdot,\cdot)$.

\begin{lem}  \label{lem:rtheta}
For any $a \in \Z_{\geq0}$ and any $i\in I$, we have
$$
r(\theta_i^{(a)}) =\sum_{t+t'=a} (\pi_iq_i)^{-tt'} \theta_i^{(t)} \otimes
\theta_i^{(t')}.
$$
\end{lem}

\begin{proof}
When $i\in I_\zero$, this is simply \cite[Lemma~1.4.2]{Lu},
which was proved directly using \eqref{eq:binom}.

Now assume $i\in I_\one$. Thanks to the identity \eqref{eq:piq2}, the
assumption of Lemma~\ref{lem:odbinom} is satisfied with $x=1\otimes
\theta_i$ and $y=\theta_i\otimes 1$. Hence this lemma follows
directly from \eqref{eq:binom-od} by the definition of the divided
power based on \eqref{eq:pi}.
\end{proof}

\begin{lem}  \label{lem:thetap2}
For any $a \in \Z_{\geq0}$ and any $i\in I$, we have
\begin{align*}
(\theta_i^{(a)}, \theta_i^{(a)})
 &=\pi_i^{a(a-1)/2}\prod_{s=1}^a
(1-\pi_i^{s} q_i^{2s})^{-1}
  \\
&= (-1)^a \pi_i^{a(a-1)/2} q_i^{-a(a+1)/2} (\pi_iq_i-q_i^{-1})^{-a}
([a]_i!)^{-1}.
\end{align*}
\end{lem}

\begin{proof}
We will only prove the first identity, as the second identity is
elementary.

The argument is similar to the proof of \cite[Lemma~1.4.4]{Lu},
which corresponds to the case for $i \in I_\zero$. We proceed by
induction. Note that the lemma holds for $a=0$ or $1$. Assume the
lemma holds for $a$ and also for $a'$. Using Lemma~\ref{lem:rtheta}
and \eqref{eq:tensorform} we have
\begin{align*}
(\theta_i^{(a+a')}, \theta_i^{(a+a')})
 &=\left[
\begin{matrix}a+a'\\a\end{matrix}\right]_i^{-1}
\Big( r(\theta_i^{(a+a')}), \theta_i^{(a)}\otimes  \theta_i^{(a')}
\Big)
 \\
&= (\pi_iq_i)^{-aa'} \left[
\begin{matrix}a+a'\\a\end{matrix}\right]_i^{-1}
 (\theta_i^{(a)},
\theta_i^{(a)})(\theta_i^{(a')}, \theta_i^{(a')})
 \\
&=(\pi_iq_i)^{-aa'} \left[
\begin{matrix}a+a'\\a\end{matrix}\right]_i^{-1}
 \pi_i^{a\choose2}\prod_{s=1}^a (1-\pi_i^{s}
q_i^{2s})^{-1} \cdot \pi_i^{a'\choose2}\prod_{s=1}^{a'}(1-\pi_i^{s}
q_i^{2s})^{-1}
 \\
&= \pi_i^{aa'+a(a-1)/2+a'(a'-1)/2}
\prod_{s=1}^{a+a'}(1-\pi_i^{s}q_i^{2s})^{-1}.
\end{align*}
Since $aa'+{a\choose2}+{a'\choose2}={{a+a'}\choose2}$, this gives
the result.
\end{proof}

\subsection{Quantum Serre relations}

We first formulate the following super analogue of the last formula
in the proof of \cite[Lemma~1.4.5]{Lu}.

\begin{lem}  \label{lem:claim}
Assume $i\in I$ is odd. Let $n\in \Z_{\geq0}$ and let $a,a',b, b'\in
\Z_{\geq0}$ such that $a+a' =b+b'=n$. Then,
\begin{align}  \label{eq:Lu12}
  \big(\theta_i^{(a)} \theta_j \theta_i^{(a')}, \theta_i^{(b)} \theta_j
\theta_i^{(b')} \big)
 = \sum \frac{(-1)^{n+1} q_i^{-\clubsuit}
\pi^{\heartsuit}
 }{(\pi q_i-q_i^{-1})^n (1-\pi_jq_j^{2}) [s]_i! [s']_i! [t]_i! [t']_i!
   },
\end{align}
where the sum is taken over all $t,t',s,s'$ in $\Z_{\geq0}$ such that %
\begin{equation} \label{eq:tsts}
t+s=b, t'+s'=b', t+t'=a, s+s'=a',
\end{equation}
and we have denoted
\begin{align*}
\spadesuit &=\frac12 \big( s(s-1) +s'(s'-1)+t(t-1)+t'(t'-1)\big ),
 \\
\clubsuit &=ss'+tt'+ts+t's'+2t's+(\spadesuit +n)+(t'+s)a_{ij},
 \\
\heartsuit &=ss'+tt'+ts+t's'+t's+(s+t')p(j) +\spadesuit.
\end{align*}
\end{lem}

\begin{proof}
We first compute by Lemma~\ref{lem:rtheta} that
\begin{align}
& r(\theta_i^{(b)}\theta_j\theta_i^{(b')})
=r(\theta_i^{(b)})r(\theta_j)r(\theta_i^{(b')})
  \notag \\
 &=\big(\sum_{t+s=b} (\pi q_i)^{-ts} \theta_i^{(t)}\otimes
 \theta_i^{(s)}\big)
 \cdot (\theta_j\otimes 1 +1\otimes \theta_j) \cdot
 \big(\sum_{t'+s'=b'} (\pi q_i)^{-t's'} \theta_i^{(t')}\otimes
 \theta_i^{(s')}\big)
  \notag \\
 &=\sum q_i^{-(ts+t's'+2t's)} q^{-s(\af_i,\af_j)} \pi^{ts+t's'+t's+sp(j)}
\theta_i^{(t)}\theta_j\theta_i^{(t')}\otimes
\theta_i^{(s)}\theta_i^{(s')}
 \notag \\
 &\quad + \sum q_i^{-(ts+t's'+2t's)} q^{-t'(\af_i,\af_j)}
\pi^{ts+t's'+t's+t'p(j)} \theta_i^{(t)}\theta_i^{(t')}\otimes
\theta_i^{(s)}\theta_j\theta_i^{(s')}.
  \label{eq:riji}
 \end{align}
Using \eqref{eq:riji} and $q^{(\af_i,\af_j)} =q_i^{a_{ij}}$, we further
compute that
\begin{align}
& (\theta_i^{(a)}\theta_j\theta_i^{(a')},
\theta_i^{(b)}\theta_j\theta_i^{(b')})
  \notag \\
& =(\theta_i^{(a)}\theta_j\otimes\theta_i^{(a')},
r(\theta_i^{(b)}\theta_j\theta_i^{(b')}) )
  \notag \\
&=\sum q_i^{-(ts+t's'+2t's)} q^{-s(\af_i,\af_j)} \pi^{t's+ts+t's'+sp(j)}
(\theta_i^{(a)}\theta_j\otimes\theta_i^{(a')},
\theta_i^{(t)}\theta_j\theta_i^{(t')}\otimes
\theta_i^{(s)}\theta_i^{(s')} )
  \notag \\
&= \sum q_i^{-(ts+t's'+2t's+s a_{ij})}   \pi^{ts+t's'+t's+sp(j)}
(\theta_i^{(a)}\theta_j, \theta_i^{(t)}\theta_j\theta_i^{(t')})
(\theta_i^{(a')}, \theta_i^{(s)}\theta_i^{(s')} ).
 \label{eq:ijiab}
\end{align}
 It follows by Lemma~\ref{lem:rtheta} that
\begin{align}
(\theta_i^{(a')}, \theta_i^{(s)}\theta_i^{(s')} )
 &=(r(\theta_i^{(a')}), \theta_i^{(s)}\otimes \theta_i^{(s')} )
 \notag \\
 &=q_i^{-ss'} \pi^{ss'} (\theta_i^{(s)}, \theta_i^{(s)})
(\theta_i^{(s')}, \theta_i^{(s')} ). \label{eq:ass'}
\end{align}
 Also using $q^{(\af_i,\af_j)} =q_i^{a_{ij}}$, we have
\begin{align}
(\theta_i^{(a)}\theta_j, \theta_i^{(t)}\theta_j\theta_i^{(t')})
 &=(r(\theta_i^{(a)})r(\theta_j), \theta_i^{(t)}\theta_j\otimes\theta_i^{(t')})
 \notag \\
 &=q_i^{-(tt'+t'a_{ij})} \pi^{tt'+t'p(j)}  (\theta_j, \theta_j)
(\theta_i^{(t)}, \theta_i^{(t)}) (\theta_i^{(t')}, \theta_i^{(t')}
). \label{eq:aijit'}
 \end{align}
Inserting \eqref{eq:ass'} and \eqref{eq:aijit'} into
\eqref{eq:ijiab}, we obtain
\begin{align}
& (\theta_i^{(a)}\theta_j\theta_i^{(a')},
\theta_i^{(b)}\theta_j\theta_i^{(b')})
\notag \\
 &= q_i^{-(ss'+tt'+ts+t's'+2t's+(s+t') a_{ij})}
\pi^{ss'+tt'+ts+t's'+t's+(s+t')p(j)} \times
 \notag \\
 &\qquad \times (\theta_j, \theta_j) (\theta_i^{(s)}, \theta_i^{(s)})
(\theta_i^{(s')}, \theta_i^{(s')} ) (\theta_i^{(t)}, \theta_i^{(t)})
(\theta_i^{(t')}, \theta_i^{(t')} ).
 \label{eq:5pair}
\end{align}
The right-hand side of \eqref{eq:5pair} can then be converted to
\eqref{eq:Lu12} using Lemma~\ref{lem:thetap2} repeatedly and noting
that
$$
\frac12 \big( s(s+1) +s'(s'+1)+t(t+1)+t'(t'+1)\big ) =\spadesuit +n.
$$
This proves the lemma.
\end{proof}

Now we are ready to state and prove the following fundamental
result, called the {\em quantum Serre relations}. Recall
$p(k;i,j)$ from \eqref{E:pkij}.
\begin{thm}  \label{th:serre}
Let $A$ be a GCM satisfying (C1)-(C5).
 For any $i \neq j $ in $I$, the following
identities hold in $\fqp$:
\begin{equation}  \label{eq:serretheta}
\sum_{a+a'=1-a_{ij}} (-1)^{a'} \pi^{p(a';i,j)} \theta_i^{(a)} \theta_j
\theta_i^{(a')} =0.
\end{equation}
\end{thm}

\begin{proof}
The strategy is to show the element on the left-hand side of
\eqref{eq:serretheta} orthogonal to $\fqp$  with respect to
$(\cdot,\cdot)$.
 For $i$ even, the same proof for \cite[Proposition~ 1.4.3]{Lu} applies
here without any change, regardless of the parity of $j$.

Now assume that $i$ is odd. Hence $p(i)=1$ and $\pi_i=\pi$. We still
proceed as in \cite{Lu} and keep track of the super signs carefully
in the meantime.

Using \eqref{eq:tsts} to get rid of $t$ and $s'$, we can rewrite
$\clubsuit$ in Lemma~\ref{lem:claim} as
\begin{align}  \label{eq:club}
\clubsuit
 &=b(b+1)/2 +b'(b'+1)/2-s(b-1)-t'(b'-1) +(t'+s)(a_{ij}+n-1).
\end{align}
(This agreed with the power of $v^{-1}$ in the corrected
\cite[Lemma~1.4.5]{Lu}.)

Recall the new indeterminate $v$ from \eqref{eq:vq}. It follows by
\eqref{eq:vq}, \eqref{eq:Zvq},  \eqref{eq:Lu12} and \eqref{eq:club}
that
\begin{align}  \label{eq:Lu1.4.5}
& \pi^{p(a';i,j)} \cdot \big(\theta_i^{(a)} \theta_j
\theta_i^{(a')}, \theta_i^{(b)} \theta_j \theta_i^{(b')} \big)
  \\
&= \frac{(-1)^{n+1} v^{b(b+1)/2 +b'(b'+1)/2}
 }{(\pi q_i-q_i^{-1})^n (1-q_j^{2}) }
\sum \frac{v^{-s(b-1)-t'(b'-1)
+(t'+s)(a_{ij}+n-1)}\sqrt{\pi}^\diamondsuit }{[s]_v! [s']_v! [t]_v!
 [t']_v!
   }
   \notag
\end{align}
where the sum is taken as in Lemma~\ref{lem:claim}. Here we have
\begin{equation} \label{eq:card}
\diamondsuit =2p(a';i,j) +2 \heartsuit -\clubsuit + \spadesuit,
\end{equation}
where $\spadesuit$ arises from the conversion \eqref{eq:Zvq}. Recall
$p(a';i,j) ={a'p(j) +\hf a'(a'-1)}$, and note that $a_{ij}=1-n$ is
an even integer by (C4). Some elementary and lengthy manipulation
using \eqref{eq:tsts} allows us to rewrite $\diamondsuit$ in
\eqref{eq:card} (in terms of $s, t',b,b'$) as
\begin{align*}
\diamondsuit &=p(j)\big(2a'+2(s+t')\big)
+a'(a'-1)+ss'+tt'+ts+t's'+2\spadesuit -n-(s+t')a_{ij}
  \\
 &=p(j)(4s+2b')
+2a'(a'-1)+b^2-b+(1-n-a_{ij})s +(n-1-a_{ij})t' -n
  \\
&\equiv 2p(j)b'+b^2-b-n \mod 4.
  \label{eq:diamond}
\end{align*}
%


Recalling $\pi^2=1$, it follows that $\sqrt{\pi}^\diamondsuit$ is
independent of $s,s',t,t'$ and can be moved to the front of $\sum$
in \eqref{eq:Lu1.4.5}.

To prove \eqref{eq:serretheta}, it suffices to show that
$\sum_{a+a'=1-a_{ij}} (-1)^{a'} \pi^{p(a',i,j)} \theta_i^{(a)}
\theta_j \theta_i^{(a')} \in {\fqp}$ is orthogonal with respect to
$(\cdot,\cdot)$ to $\theta_i^{(b)} \theta_j \theta_i^{(b')}$ for all
$b,b'$ such that $b+b'=1-a_{ij}$. (These elements $\theta_i^{(b)}
\theta_j \theta_i^{(b')}$ span ${\fqp}_{\af_j+ (1-a_{ij})\af_i}.$)
To that end, recalling \eqref{eq:Lu1.4.5} and noting $a_{ij}+n-1=0$,
it remains to verify the identity
$$
\sum (-1)^{s+s'} v^{-s(b-1) -t'(b'-1)} ([s]_v! [s']_v! [t]_v!
 [t']_v!)^{-1} =0,
 $$
where the sum is taken as in Lemma~\ref{lem:claim} again. We can
factorize the sum on the left hand side as
$$
\Big( \sum_{t+s=b} (-1)^s v^{-s(b-1)}[s]_v! [t]_v!)^{-1} \Big)
\Big( \sum_{t'+s'=b'} (-1)^{s'} v^{-s'(b'-1)}[s']_v! [t']_v!)^{-1}
\Big).
$$
Since at least one of $b$ and $b'$ is positive, one of the two
factors must be zero thanks to a classical binomial identity (cf.
\cite[1.3.4(a)]{Lu}).
\end{proof}

\subsection{The bar involution}\label{SS:barinv}

It is shown in \cite{BKM} that the integrable modules of the quantum
Kac-Moody superalgebras have the same characters as their classical
counterparts, generalizing the earlier work of \cite{Lu1}. Now based
on this fact and proceeding as in \cite[33.1]{Lu}, we can establish
Theorem~\ref{th:KMpres}(a) below as a super analogue of
\cite[Theorem~33.1.3]{Lu} (which is reformulated as
Theorem~\ref{th:KMpres}(b) below). Recall the algebras $\fqm, \fq,
\fqp$ were introduced in earlier subsections.

\begin{thm}  \label{th:KMpres}
There exist isomorphisms of $\Q(q)$-algebras:
\begin{enumerate}
\item[(a)]  $\fqp/(\pi+1)\cong\fqm$,
\item[(b)]  $\fqp/(\pi-1)\cong\fq$.
\end{enumerate}
\end{thm}

\begin{cor} $\Rad(\cdot,\cdot)$ is generated by Serre relations.
\end{cor}

\begin{proof} When $\pi=1$ this is proved in \cite[33.1]{Lu}. Using \cite{BKM}, a similar argument proves this in the case $\pi=-1$.
\end{proof}

\begin{rem}
Theorem~\ref{th:KMpres}(a) identifies half of the quantum Kac-Moody
superalgebra associated to the generalized Cartan matrix $A$ in
\S\ref{subsec:GCM} with the quotient of a free algebra by the
radical of an analogue of Lusztig's bilinear form. A similar result
for quantum $\mathfrak{osp}(1|2n)$ was also obtained in \cite{Ya}
using a different normalization of bilinear form such that
$(\theta_i,\theta_i)=1$ for $i$ odd and with additional signs
appearing in the definition of the form on $\fqm\otimes\fqm$ (see also \cite{Gr}).
\end{rem}

We will write $\fap$ for the $\A^{\pi}$-subalgebra of $\fqp$
generated by the divided powers $\te_i^{(k)}=\te_i^k/[k]_i!$ (see
\eqref{eq:pi} for $[k]_i$), subject to the quantum Serre relation
\eqref{eq:serretheta}.

\begin{dfn}\label{D:Bar}
Under the assumption (C1)-(C6) for the GCM $A$, we define the bar
involution $\bar{\quad}:\fqp\to\fqp$ by letting
\begin{equation}  \label{eq:bar}
\ov{\pi}=\pi, \qquad \ov{q} =\pi q^{-1}, \qquad \ov{\te_i} =\te_i
\quad (\forall i\in I).
\end{equation}
\end{dfn}
We have that
$$
\ov{q_i}=\pi^{s_i}q_i^{-1},
$$
and a calculation gives
$$
\ov{[k]_i}=\pi_i^{(k-1)}\pi^{s_i(k-1)}[k]_i=\pi^{(s_i+p(i))(k-1)}[k]_i.
$$
It now follows from (C6) that the quantum integers $[k]_i$ are
bar-invariant, and so the divided powers $\te_i^{(k)}$ are
bar-invariant as well. Thus, this induces a bar-involution
$\bar{\quad}:\fap\to\fap$.

\section{Spin quiver Hecke algebras}\label{S:sqHa}

\subsection{Generators and relations}\label{subsec:gen&rel}

Fix an $\ell\times\ell$ GCM $A$ as in $\S$\ref{subsec:GCM}, and
continue to assume (C6) as usual. Let $n\in\Z_{\geq0}$, and assume
for $\nu\in Q^+$ (see \eqref{E:RootLattice}) that
$\nu=n_1\af_1+\cdots+n_\ell\af_\ell$ and $n_1+\cdots+n_\ell=n$ (i.e.
$\nu$ has \emph{height} $\height(\nu)=n$). Let $I^\nu\subset I^n$ be
the $S_n$-orbit of the element
$$
(\underbrace{1,\ldots,1}_{n_1},\ldots,\underbrace{\ell,\ldots,\ell}_{n_\ell}),
$$
where $S_n$ acts on $I^n$ by place permutation:
$w\cdot(i_1,\ldots,i_n)=(i_{w(1)},\ldots,i_{w(n)})$. Equivalently,
\begin{align}\label{E:Inu}
I^\nu=\{ \ui=(i_1,\ldots,i_n)\in I^n \,|\, \af_{i_1}+\cdots+\af_{i_n}=\nu \}
\end{align}

We now define an algebra based on the data above in terms of
generators and relations. When $I_\one=\emptyset$, the algebra is
nothing but the quiver Hecke algebra of Khovanov-Lauda and Rouquier
\cite{KL1, Ro1}. In the general case $I_\one\neq\emptyset$ we are
considering, these algebras were recently defined in \cite{KKT}. We
refer to them as \emph{spin} quiver Hecke algebras as explained in
the introduction. In the special case when $I=I_\one$ is a
singleton, the algebra is the spin nilHecke algebra, a nil version
of the spin Hecke algebra first introduced in \cite[3.3]{Wa}. The
spin quiver Hecke algebra is defined to be
\[
\Hm=\bigoplus_{\nu\in Q^+}\Hm(\nu),
\]
where $\Hm(\nu)$ is the unital $\bbk$-algebra, with identity
$1_\nu$, given by generators and relations as described below.

The generators of $\Hm(\nu)$ are
\[
\{e(\ui)|\ui\in I^\nu\}\cup\{y_1,\ldots,
y_n\}\cup\{\tau_1,\ldots,\tau_{n-1}\}.
\]
We refer to the $e(\ui)$ as \emph{idempotents}, the $y_r$ as
\emph{(skew) Jucys-Murphy elements}, and the $\tau_r$ as \emph{intertwining
elements}. Indeed, these generators are subject to the following
relations for all $\ui,\uj\in I^\nu$ and all admissible $r,s$:
{\allowdisplaybreaks
\begin{align}
\label{E:Relations1}e(\ui)e(\uj)&=\dt_{\ui,\uj}e(\ui);\\
\label{E:Relations2}\sum_{\ui\in I^\nu}e(\ui)&=1_\nu;\\
\label{E:Relations3}y_re(\ui)&=e(\ui)y_r;\\
\label{E:Relations4}\tau_re(\ui)&=e(s_r\cdot\ui)\tau_r;\\
\label{E:Relations5}y_ry_se(\ui)&=(-1)^{p(i_r)p(i_s)}y_sy_re(\ui);\\
\label{E:Relations6}
 \tau_ry_se(\ui)
 &=(-1)^{p(i_r)p(i_{r+1})p(i_s)}y_s\tau_re(\ui)&\mbox{if }s\neq r,r+1;\\
\label{E:Relations7}
 \tau_r\tau_se(\ui)
 &=(-1)^{p(i_r)p(i_{r+1})p(i_s)p(i_{s+1})}\tau_s\tau_re(\ui)
 &\mbox{if }|s-r|>1;\\
\label{E:Relations8}
 \tau_ry_{r+1}e(\ui)&=\begin{cases}((-1)^{p(i_r)p(i_{r+1})}y_r\tau_r+1)e(\ui)\\
 (-1)^{p(i_r)p(i_{r+1})}y_r\tau_re(\ui)\end{cases}
 &\begin{array}{r}i_r=i_{r+1},\\i_r\neq
  i_{r+1};\end{array}\\
\label{E:Relations9}
 y_{r+1}\tau_re(\ui)&=\begin{cases}((-1)^{p(i_r)p(i_{r+1})}\tau_ry_r+1)e(\ui)\\
 (-1)^{p(i_r)p(i_{r+1})}\tau_ry_re(\ui)\end{cases}
 &\begin{array}{r}i_r=i_{r+1},\\i_r\neq
                        i_{r+1}.\end{array}
\end{align}
 }
Additionally, the intertwining elements satisfy the quadratic relations
\begin{align}\label{E:Quadratic Relations}
\tau_r^2e(\ui)=\calQ_{i_r,i_{r+1}}(y_r,y_{r+1})e(\ui)
\end{align}
for all $1\leq r\leq n-1$, and the braid-like relations
\begin{align}\label{E:Braid Relations}
(\tau_r&\tau_{r+1}\tau_r-\tau_{r+1}\tau_r\tau_{r+1})e(\ui)\\
\nonumber =&
 \begin{cases}\left(\frac{\calQ_{i_r,i_{r+1}}(y_{r+2},y_{r+1})
 -\calQ_{i_r,i_{r+1}}(y_r,y_{r+1})}{y_{r+2}-y_r} \right)
  e(\ui)&\mbox{if }i_r=i_{r+2}\in I_\zero,\\
    (-1)^{p(i_{r+1})}(y_{r+2}-y_r)
  \left(\frac{\calQ_{i_r,i_{r+1}}(y_{r+2},y_{r+1})
  -\calQ_{i_r,i_{r+1}}(y_r,y_{r+1})}{y_{r+2}^2-y_r^2}
  \right)e(\ui)&\mbox{if }i_r=i_{r+2}\in I_\one,\\
    0&\mbox{otherwise,}\end{cases}
\end{align}
for $1\leq r\leq n-2$. A subtle point, explained in \cite{KKT}, is
that in the case $i_r=i_{r+2}\in I_\one$ above, $y_r^2$ and
$y_{r+2}^2$ are even and, consequently, commute with
$y_1,\ldots,y_n$.  Therefore, there is no ambiguity in the
corresponding formula.

Finally, this algebra is bi-graded, with $\Z$-grading given by
\begin{align}\label{E:ZGrading}
\deg e(\ui)=0,\;\;\;\deg
y_re(\ui)
=(\af_{i_r},\af_{i_r}),\andeqn\deg\tau_re(\ui)=-(\af_{i_r},\af_{i_{r+1}}),
\end{align}
and $\Z_2$-grading given by
\begin{align}\label{E:Z2Grading}
p(e(\ui))=0,\;\;\; p(y_re(\ui))
=p(i_r),\andeqn p(\tau_re(\ui))=p(i_r)p(i_{r+1}).
\end{align}

\subsection{An automorphism and antiautomorphism}\label{SS:automorphism}

The following two propositions can be verified directly by
definition.

\begin{prp}\label{P:Automorphismphi}
There is a unique $\bbk$-linear automorphism $\phi:\Hm(\nu)\to
\Hm(\nu)$ given by
\begin{align*}
\phi(e(\ui))=e(w_0\cdot\ui),\;\;   \phi(y_r)& =y_{n-r+1},
 \\
\phi(\tau_re(\ui))= (-1)^{1+p(i_r)r(i_{r+1})}
&\tau_{n-r}e(s_rw_0\cdot\ui),
\end{align*}
where $w_0\in S_n$ is the longest element.
\end{prp}

\begin{prp}\label{P:Anti-AutomorphismPsi}
There is a unique $\bbk$-linear anti-automorphism $\psi:\Hm(\nu)\to
\Hm(\nu)$ defined by $\psi(e(\ui))=e(\ui)$, $\psi(y_r)=y_r$, and
$\psi(\tau_s)=\tau_s$ for all $\ui\in I^\nu$ and admissible $r,s$.
\end{prp}

\subsection{The rank 1 case}\label{subsec:rk1}

Note that the diagram of a single node $\xy(0,0)*{\bullet};\endxy$
corresponds to Lie superalgebra $\osp(1|2)$. In this section, we
will also consider the diagram $\xy(0,0)*{\circ};\endxy$ which
corresponds to $\Sl(2)$. These are quivers with compatible
automorphism (the trivial one) and the corresponding GCM is $A=(2)$,
and there is only one polynomial $\calQ(u,v)\equiv 0$.

In the even case $\circ$, the quiver Hecke algebra is the nilHecke
algebra $\NH_n$ generated by subalgebras $\bbk[Y]$ for
$Y=\{y_1,\ldots,y_n\}$ and the nil-Coxeter algebra $\NC_n$. The
nil-Coxeter algebra is generated by the divided difference operators
$\del_r^+$, $1\leq r<n$, which are subject to the relations
\begin{align}\label{E:nilCox}
\nonumber(\del_r^+)^2&=0,\\
\del_r^+\del_s^+&=\del_s^+\del_r^+,\;\;\;|r-s|>1,\\
\nonumber\del_r^+\del_{r+1}^+\del_r^+&=\del^+_{r+1}\del_r^+\del^+_{r+1}.
\end{align}
The nil-Coxeter algebra act faithfully on $\bbk[Y]$ via
\begin{align}\label{E:DivDiff+}
\del_r^+\mapsto \frac{1-s_r}{y_{r+1}-y_{r}},
\end{align}
where the simple transposition acts on polynomials by permuting the
variables as usual.

For the odd case, the corresponding spin (or odd) nilHecke algebra
$\NHm_n$ is a nil version of the spin Hecke algebra in
\cite[3.3]{Wa} and we follow the presentation in \cite[2.2]{EKL}
below. Let $\bbk[Y]^-$ be a skew-polynomial ring in $n$ variables
(that is, the quotient of the free algebra on $Y=\{y_1,\ldots,y_n\}$
by the relation $y_ry_s=-y_sy_r$ for $r\neq s$). Then, $\NHm_n$ is
generated by subalgebras $\bbk[Y]^-$ and the spin (or odd)
nil-Coxeter algebra $\NCm_n$. The spin nil-Coxeter algebra is
generated by an odd analogue of divided difference operators
$\del_r^-$, $1\leq r<n$, which are subject to the relations
\begin{align}\label{E:spinnilCox}
\nonumber(\del_r^-)^2&=0,\\
\del_r^-\del_s^-&=-\del_s^-\del_r^-,\;\;\;|r-s|>1,\\
\nonumber\del_r^-\del_{r+1}^-\del_r^-&=\del^-_{r+1}\del_r^-\del^-_{r+1}.
\end{align}
The spin nil-Coxeter algebra acts faithfully on $\bbk[Y]^-$ as in
the even case, where
\begin{align}\label{E:DivDiff-1}
\del_r^-(y_s)=
\begin{cases}1&\mbox{if }s=r,r+1,\\0&\mbox{otherwise,}\end{cases}
\end{align}
and the action of $\del_r$ extends to $\bbk[Y]^-$ using the
\emph{Leibnitz rule}: for $f,g\in\bbk[Y]^-$,
\begin{align}\label{E:DivDiff-2}
\del_r^-(fg)=\del_r^-(f)g+s_r(f)\del_r^-(g).
\end{align}
Here the action of $S_n$ on $\bbk[Y]^-$ is given by
$s_r(y_k)=-y_{s_r(k)}$.

\subsection{A basis theorem}\label{subsec:BasisThm}
We now show that $\Hm(\nu)$ satisfies the PBW property. More
precisely, we have the following.

\begin{thm}\label{T:PBW}
For each $w\in S_n$, fix a reduced expression $w=s_{i_1}\cdots
s_{i_r}$ for $w$, and let $\tau_w=\tau_{i_1}\cdots\tau_{i_r}$. Set
\begin{align}\label{E:PBW Basis}
\mathcal{B}=\{\tau_w y_1^{a_1}\cdots y_n^{a_n}e(\ui)\,|\,\ui\in
I^\nu,\;(a_1,\ldots,a_n)\in\Z_{\geq0}^n,\;w\in S_n\}.
\end{align}
Then, $\Hm(\nu)$ is free over $\bbk$ with basis $\mathcal{B}$.
\end{thm}

Note that by \eqref{E:Braid Relations} this depends on the choice of
reduced expression. However, we have the following proposition, the
proof of which is almost identical to \cite[Proposition 2.5]{BKW}.

\begin{prp}\label{P:Dependence on Reduced Words} Let $\ui\in I^\nu$ and $w\in
S_n$. Assume that $w=s_{k_1}\cdots s_{k_t}$ is a presentation of $w$ as a
product of simple transpositions.
\begin{enumerate}
\item[(a)]
If the presentation $w=s_{k_1}\cdots s_{k_t}$ is reduced, and
$w=s_{\ell_1}\cdots s_{\ell_t}$ is another reduced presentation of
$w$, then
\[
\tau_{k_1}\cdots\tau_{k_t}e(\ui)
=\pm\tau_{\ell_1}\cdots\tau_{\ell_t}e(\ui)+\sum_{u<w}\tau_uf_u(y)e(\ui)
\]
where the sum is over $u<w$ in the Bruhat order, $f_u(y)$ is a
polynomial in $y_1,\ldots,y_n$, and
$\deg\tau_uf_u(y)e(\ui)=\deg\tau_we(\ui)$ for all $u$.

\item[(b)]
If the presentation $w=s_{k_1}\cdots s_{k_t}$ is not reduced, then
$\tau_{k_1}\cdots\tau_{k_t}e(\ui)$ can be written as a linear
combination of words of the form
$\tau_{k_{r_1}}\cdots\tau_{k_{r_s}}f(y)e(\ui)$, such that $1\leq
r_1<\cdots<r_s\leq t$, $s<t$, $w=s_{k_{r_1}}\cdots s_{k_{r_s}}$ is
reduced, $f(y)$ is a polynomial in $y_1,\ldots y_n$, and
\[
\deg\tau_{k_1}\cdots\tau_{k_t}e(\ui)
=\deg\tau_{k_{r_1}}\cdots\tau_{k_{r_s}}f(y)e(\ui).
\]
\end{enumerate}
\end{prp}

We now turn to the proof of Theorem \ref{T:PBW}. To this end, let
$\curlyP^-(\nu)$ be the subalgebra of $\Hm(\nu)$ generated by the
elements $y_1,\ldots,y_n$ and $e(\ui)$, $i\in I^\nu$. Fix a total
ordering on $I$, and define a family of polynomials
$\mathsf{P}=(\mathsf{P}_{ij}(u,v))_{i,j\in I}$ by the formula
\begin{align}\label{E:Pmatrix}
\mathsf{P}_{ij}(u,v)=\begin{cases}0&\mbox{if }i=j,\\
                        \calQ_{ij}(u,v)&\mbox{if }i<j,\\
                        1&\mbox{if }i>j.\end{cases}
\end{align}
Then, Theorem \ref{T:PBW} follows immediately from the following
proposition.

\begin{prp}\label{P:PolynomialRep}
There is a faithful action of the algebra $\Hm(\nu)$ on
$\curlyP^-(\nu)$, where $y_s$ and $e(\ui)$ act by left
multiplication ($s=1,\ldots,n$, $\ui\in I^\nu$), and
$$
\tau_re(\ui)\mapsto
 \begin{cases}\del_r^+e(\ui)&\mbox{if }i_r=i_{r+1}\in I_\zero,\\
 \del_r^-e(\ui)&\mbox{if }i_r=i_{r+1}\in I_\one,\\
 \mathsf{P}_{i_r,i_{r+1}}(y_{r+1},y_{r})e(s_r\cdot\ui)s_r
 &\mbox{otherwise,}
 \end{cases}
$$
where $S_n$ acts on $\curlyP^-(\nu)$ according to
\begin{align}\label{E:SymAction}
s_k(y_re(\ui))=(-1)^{p(i_k)p(i_{k+1})p(i_r)}y_{s_k(r)}e(s_k\cdot\ui).
\end{align}
\end{prp}

\begin{proof}
The proof can be adapted from either \cite[Proposition 2.3, Theorem
2.5]{KL1}, or \cite[Proposition 3.12]{Ro1}.
\end{proof}

\subsection{The center}\label{subsec:Center}

For $\nu=\sum_{i\in I}n_i\af_i$, let $n_\one=\sum_{i\in I_\one}n_i$.
Let
\begin{align}\label{E:Ynu}
Y_\nu=\{y_r^{1+p(i_r)}e(\ui)|r=1,\ldots,n,\ui\in I^\nu\},
\end{align}
and let $\Sym(\nu)=\bbk[Y_\nu]^{S_n}\subset\curlyP^-(\nu)$, where the
symmetric group $S_n$ acts on $\curlyP^-(\nu)$ by \eqref{E:SymAction}.
Then, we obtain the following by combining the proofs of
\cite[Proposition~ 2.7]{KL1} and \cite[Proposition~ 2.15]{EKL}.

\begin{prp}\label{P:Center}
The center of $\Hm(\nu)$ is $\Sym(\nu)$. Moreover, $\Hm(\nu)$ is
free of rank $2^{n_\one}(n!)^2$ over $\Sym(\nu)$.
\end{prp}

\section{Categorification of quantum Serre relations} \label{sec:Representations}

\subsection{Module categories}\label{subsec:ModCat}

Throughout the rest of the paper, we will primarily work inside the
abelian category $\Modm(\nu):=\Mod\Hm(\nu)$ of finitely generated
$(\Z\times\Z_2)$-\emph{graded} left $\Hm(\nu)$-modules, with
morphisms being $\Hm(\nu)$-homomorphisms that preserve \emph{both} the
$\Z$-grading and $\Z_2$-grading. Let $\Hom_\nu(M,N)$ denote the
$\bbk$-vector space of morphisms between $M$ and $N$.

Note that $\Hm(\nu)$ is \emph{almost} positively $\Z$-graded,  cf.
\cite[p.24]{KL1}. Indeed, it is nontrivial only in degrees greater
than, or equal to, $-\sum_i n_i(n_i-1)$, for $\nu=\sum_{i\in
I}n_i\af_i$. In particular, for any $M\in\Modm(\nu)$, we may define
its {\em $(q,\pi)$-dimension}
\begin{align}\label{E:qpDim}
\dim_q^\pi M=\sum_{a\in\Z}(\dim M_\zero[a] +\pi\dim
M_\one[a])q^a\in\Z((q))^{\pi}.
\end{align}
The specialization $\pi\mapsto 1$ recovers the usual graded
dimension of the module $M$, while $\pi\mapsto-1$ produces the
graded \emph{super}dimension. Additionally, we define the
\emph{graded character} of $M\in\Modm(\nu)$ by
\begin{align}\label{E:FormalChar}
\ch_q^\pi M=\sum_{\ui\in I^\nu}(\dim_q^\pi e(\ui)M)\cdot \ui.
\end{align}
This is an element of the free $\Z((q))^\pi$-module with basis
labelled by $I^\nu$. Define
\begin{align}\label{E:CharSpecialization}
\ch_q^+M=\ch_q^\pi M|_{\pi=1}\andeqn \ch_q^-M=\ch_q^\pi M|_{\pi=-1}.
\end{align}
Note that we may have $\ch_q^-M=0$.

There exists a parity shift functor
$$\Pi:\Modm(\nu)\longrightarrow \Modm(\nu)
$$
which is defined on an object $M$ by letting $\Pi M=M$ as
$\Hm(\nu)$-modules, but with $(\Pi M)_\zero=M_\one$ and $(\Pi
M)_\one=M_\zero$. We set
\begin{align}
\label{E:SuperHom}\Homm_\nu(M,N)=\Hom_\nu(M,N)\oplus\Hom_\nu(M,\Pi N).
\end{align}
Then, $\Homm_\nu(M,N)$ is $\Z_2$-graded, with
$$
\Homm_\nu(M,N)_\zero=\Hom_\nu(M,N), \quad
\Homm_\nu(M,N)_\one=\Hom_\nu(M,\Pi N).
$$

For each $a\in\Z$, there is a functor
$\texttt{q}:\Modm(\nu)\to\Modm(\nu)$ which shifts the $\Z$-grading:
$\texttt{q}M = M\{1\}$ where, for $a\in\Z$, the module $M\{a\}$
equals $M$ as a $\Hm(\nu)$-module, but has $k$th graded component
$(M\{a\})[k]=M[k-a]$ (the $(k-a)$th graded component of $M$). Define
the space of $\Pi$-twisted \emph{enhanced} homomorphisms
\begin{align}\label{E:HOM}
\HOMm_\nu(M,N)=\bigoplus_{a\in\Z}\Homm_\nu(M,\Pi^a N\{a\}).
\end{align}

The category $\Modm(\nu)$ contains the full subcategories
$\Repm(\nu):=\Rep\Hm(\nu)$ and $\Projm(\nu):=\Proj\Hm(\nu)$ of
finite dimensional and finitely generated projective modules,
respectively. The category $\Repm(\nu)$ is abelian, while
$\Projm(\nu)$ is additive. The corresponding Grothendieck groups
$[\Repm(\nu)]$ and $[\Projm(\nu)]$ are naturally $\A^{\pi}$-modules.
The functors $\Pi$ and $\texttt{q}$ induce an action of $\A^{\pi}$
on both $[\Projm(\nu)]$ and $[\Repm(\nu)]$ via
$q[M]=[\texttt{q}M]=[M\{1\}]$ and $\pi[M]=[\Pi M]$, where $M$ is an
object in the relevant category (recall that $\Hom_\nu(M,N)$ is the
space of morphisms, and NOT $\Homm_\nu(M,N)$).

Recall the anti-automorphism $\psi$ from
Proposition~\ref{P:Anti-AutomorphismPsi}. For $P\in\Projm(\nu)$,
define
$$
P^\#=\HOMm_\nu(P,\Hm(\nu))^\psi,
$$
where for a finitely generated graded \emph{right} (resp.
\emph{left}) $\Hm(\nu)$-module $M$, $M^\psi$ is the \emph{left}
(resp. \emph{right}) module obtained by twisting the action on $M$
by the anti-automorphism $\psi$: $y.m=m.\psi(y)$ (resp.
$m.y=\psi(y).m$) for $m\in M$ and $y\in\Hm(\nu)$. Let $\mathbf{1}$
denote the unique simple $\Hm(0)$ module. It follows by
\eqref{E:HOM} that
\begin{align}
(\texttt{q}\mathbf{1})^{\#}\cong\Pi\texttt{q}^{-1}\mathbf{1}.
 \label{eq:barq}
\end{align}

For each $\ui\in I^\nu$, define the projective module
\begin{align}\label{E:Pui}
P_\ui=\Hm(\nu)e(\ui).
\end{align}
Then, the mapping which sends $e(\ui)$ to the natural embedding
$\{P_\ui\hookrightarrow\Hm(\nu)\}$ defines an isomorphism
\begin{align}\label{E:Pound}
P_\ui\cong P_\ui^{\#}.
\end{align}
Moreover,  $(P_\ui\{a\})^{\#}\cong \Pi^a
P_{\ui}\{-a\}$. Let
\begin{align}\label{E:ProjBar}
\bar{\quad}:[\Projm(\nu)]\longrightarrow [\Projm(\nu)]
\end{align}
be the $\Z$-linear involution: $\ov{\pi}=\pi$, $\ov{q}=\pi q^{-1}$,
and $\ov{[P]}=[P^{\#}]$.

There is a bilinear form
\begin{align}\label{E:ProjmBilinForm}
(\cdot,\cdot):[\Projm(\nu)]\times[\Projm(\nu)]\longrightarrow
\Z((q))^{\pi}
\end{align}
given by
\begin{align}\label{E:ProjFormDfn}
([P],[Q])=\dim_q^\pi (P^\psi\otimes_{\Hm(\nu)}Q).
\end{align}
Note that $(P_\ui,P_\uj)=\dim_q^\pi e(\uj)\Hm(\nu)e(\ui)$. For
future reference, we also note that natural form on
$[\Proj\,\Hm(\mu)\otimes\Hm(\nu)]$ is given by
\begin{align}\label{E:cattensorform}
 ([P\otimes Q],[P'\otimes Q'])=\dim_q^\pi((P\otimes
Q)^\psi\otimes_{\Hm(\mu)\otimes\Hm(\nu)}(P'\otimes Q')).
\end{align}

\begin{lem}\label{L:EquivProjForm}
For $P,Q\in\Projm(\nu)$, we have
$
([P],[Q])=\dim_q^\pi\HOMm_\nu(P^{\#},Q).
$
\end{lem}

\begin{proof}
For $P,Q\in\Projm(\nu)$, we compute
\begin{align*}
([P^{\#}],[Q])&=\dim_q^\pi(P^{\#})^\psi\otimes_{\Hm(\nu)}Q\\
    &=\dim_q^\pi\HOMm_\nu(P,\Hm(\nu))\otimes_{\Hm(\nu)}Q\\
    &=\dim_q^\pi\HOMm_\nu(P,Q).
\end{align*}
The lemma is proved.
\end{proof}

The simple objects in $\Modm(\nu)$ belong to the category
$\Repm(\nu)$. Let $\Symz(\nu)$ be the unique maximal graded ideal in
$\Sym(\nu)$, spanned by $S_n$-invariant polynomials without constant
term. We have the following.

\begin{prp}\label{P:FinManySimples}
$\Symz(\nu)$ acts by 0 on any simple module. Hence, there are only
finitely many simple $\Hm(\nu)$-modules up to isomorphism and
grading/parity shift.
\end{prp}

\begin{proof}
Recall the commuting family of elements $Y_\nu$ from \eqref{E:Ynu}.
These elements are positively graded, and therefore act nilpotently
on any finite dimensional module. Given a simple module,  we can
find a simultaneous eigenvector for the action of $Y_\nu$
(necessarily with eigenvalue 0). By definition, $\Symz(\nu)$ acts as
0 on this vector. But, by Proposition \ref{P:Center}, $\Symz(\nu)$
is contained in the center of $\Hm(\nu)$ and so this determines its
action on the entire simple module. It now follows that every simple
module factors through the $2^{n_\one}(n!)^2$-dimensional algebra
$\Hm(\nu)/\Symz(\nu)\Hm(\nu)$, so there can only be finitely many.
\end{proof}

For each $M\in\Repm(\nu)$, we associate
$M^{\circledast}=\HOMm_\bbk(M,\bbk)$ ($\Pi$-twisted linear maps),
where $\bbk=\bbk_\zero$. There is an $\Hm(\nu)$-action on
$M^\circledast$ given by $(x f)(m)=f(\psi(x)m)$, for $x\in\Hm(\nu)$,
$f\in M^\circledast$, and $m\in M$, making $M^\circledast$ an object
in $\Repm(\nu)$. Using the definitions, we deduce that for any
simple module $L$, $L^\circledast\cong\Pi^aL\{a\}$ for some
$a\in\Z$. Replacing $L$ with $\Pi^aL\{a\}$, we obtain $L^\circledast\cong
L$.

This duality defines an involution
\begin{align}\label{E:RepBar}
\bar{\quad}:[\Repm(\nu)]\longrightarrow [\Repm(\nu)]
\end{align}
given by $\ov{\pi}=\pi$, $\ov{q}=\pi q^{-1}$, and $\ov{[M]}=[M^{\circledast}]$.

Define the $\A^\pi$-sesquilinear (i.e., antilinear in first
variable) Cartan pairing
\begin{align}\label{E:Cartanpairing}
\lan\cdot,\cdot\ran:[\Projm(\nu)]\times[\Repm(\nu)]\longrightarrow
\A^\pi
\end{align}
given by
\begin{align}\label{E:Cartanpairingdfn}
\lan[P],[M]\ran&=\dim_q^\pi \HOMm_\nu(P^\#,M).
\end{align}

For each simple module $L\in\Repm(\nu)$, there exists a (unique up
to $(\Z\times\Z_2)$-homogeneous isomorphism) projective
indecomposable cover $P_L\in\Projm(\nu)$. The modules $L$ and $P_L$
are dual with respect to \eqref{E:Cartanpairing}. Moreover, we have
$$
P_L^\#\cong P_{L^\circledast}.
$$

\subsection{More rank 1 calculations}

In this subsection, we will consider the algebra $\Hm(n\af_i)$,
which is isomorphic to either the nilHecke algebra or spin nilHecke
algebra, depending on whether $i$ is even or odd, respectively.
Accordingly, we will revert to the notation of $\S$\ref{subsec:rk1}
to better facilitate comparison with \cite[Section 3]{La} and
\cite[Section 2]{EKL}. Indeed, this section is essentially a review
of results from \cite{La,EKL} which we need, adapted for
computations with \emph{left} modules (cf. \eqref{E:Projective}
below) and enhanced by the insertion of $\pi$ to keep track of the
parity. This amounts to reading their diagrams from top to bottom,
as opposed to bottom to top (alternatively, applying the
antiautomorphism $\psi$). Additional signs also appear, caused by
\eqref{E:Relations8} which differs from the corresponding formula in
\emph{loc. cit.}, but is commonly used in the literature, cf.
\cite{BK2}.

The even and odd cases will be treated simultaneously by writing
$\del^\pm_r$ as necessary for the relevant divided difference
operator. Of course, both $\NH_n(=\NH^+_n$) and $\NHm_n$ can be
viewed as $(\Z,\Z_2)$-graded algebras as in the
$\S$\ref{subsec:gen&rel}. Our first observation is that
$(q,\pi)$-dimension of these algebras is
\begin{align}\label{E:dimqpNH}
\dim_q^\pi\NH^{\pm}_n=\frac{(\pi_i q_i)^{-{n\choose 2}}[n]_i!}{(1-\pi_i q_i^2)^n}.
\end{align}
This can be established exactly as in \cite[$\S3.1$]{La}, keeping
track of the parity when the superscript is $-$ (see also
\cite[Proposition~ 2.11]{EKL}).

Following \cite[(2.12)]{EKL}, define the algebra of (spin) symmetric
polynomials
\begin{align}\label{E:Lambda}
\Ld_n^\pm=\bigcap_{r=1}^{n-1}\ker(\del_r^\pm)\subset\bbk[Y]^\pm.
\end{align}
This algebra has a $(\Z\times\Z_2)$-homogeneous basis given by the
elementary (odd) symmetric functions
$\ep_\ld^\pm=\ep^\pm_{\ld_1}\cdots\ep^\pm_{\ld_h}$, where
$\ld=(\ld_1\geq \cdots\geq\ld_h)$ is a partition of $n$ and
\begin{align}\label{E:elemsym}
\ep_k^\pm(y_1,\ldots,y_n)=\sum_{1\leq r_1<\ldots<r_k\leq
n}(\pm1)^{r_1+\cdots+r_k-k}y_{r_1}\cdots y_{r_k},
\end{align}
see \cite[(2.21), Lemma 2.3, Remark 2.4]{EKL}. Then, $\ep_k^\pm$ has
bi-degree $(2k,p(i)k)\in\Z\times\Z_2$, and a straightforward
computation analogous to \cite[(2.18)]{EKL} proves that
\begin{align}\label{E:Ldqpidim}
\dim_q^\pi\Lambda^\pm_n=\frac{(\pi_i q_i)^{-{n\choose
2}}}{[n]_i!(1-\pi_i q_i^2)^n}
\end{align}

Below is a $\pi$-enhanced version of
\label{P:nilHecke}\cite[Proposition 3.5]{La} and \cite[Corollary
2.14]{EKL}.
\begin{prp}
The natural action of $\NH^\pm_n$ on $\bbk[Y]^\pm$ defines an
isomorphism
$$
\NH^\pm_n\cong\End_{\Ld_n^\pm}(\bbk[Y]^\pm)\cong\mathrm{Mat}\left((\pi_i
q_i)^{n\choose2}[n]_i!,\,\Ld_n^\pm\right).
$$
\end{prp}

Define the Demazure operator
\begin{align}\label{E:Demazure}
\ov{\del}_r^\pm=-\del_r^\pm y_r.
\end{align}
Then, using \eqref{E:Relations8},
$$
(\ov{\del}_r^\pm)^2=(-\dpm_ry_r)(-\dpm_ry_r)
=\dpm_r(\pm\dpm_ry_{r+1}-1)y_r=-\dpm_ry_r=\ov{\del}_r^\pm.
$$
Moreover, it is straightforward to verify that these elements
satisfy the type $A$ braid relations:
\begin{align*}
\ov{\del}^\pm_r\ov{\del}^\pm_s&=\ov{\del}^\pm_s\ov{\del}^\pm_r,&|r-s|>1,\\
\ov{\del}^{\pm}_r\ov{\del}^{\pm}_{r+1}\ov{\del}^{\pm}_r
&=\ov{\del}^{\pm}_{r+1}\ov{\del}^{\pm}_r\ov{\del}^{\pm}_{r+1}.&
\end{align*}
In particular, for each $w\in S_n$, the element $\ov{\del}_w^\pm$ is
well-defined in terms of any reduced expression of $w$.  Set
\begin{align}\label{E:idempotent}
\epm_n=\ov{\del}_{w_0}^\pm.
\end{align}

Since the spin nil-Coxeter algebra satisfies the braid relations for
the spin symmetric group as opposed to the standard braid relations,
the elements $\del^-_w$, $w\in S_n$ are only well-defined up to sign
and, therefore, we will fix reduced expressions $w=s_{r_1}\cdots
s_{r_k}$ for each element $w\in S_n$ and define
$\del^-_w=\del_{r_1}^-\cdots\del_{r_k}^-$ to remove this ambiguity.
For the moment, let $w_0\lan n\ran$ denote the longest word in
$S_n$. We define the fixed reduced expression for $w_0\lan n\ran$
inductively by $w_0\lan1\ran=1$, and $w_0\lan n\ran=s_1s_2\cdots
s_{n-1}w_0\lan n-1\ran $, for $n>1$. Then, we have a well-defined
element
\begin{align}\label{E:delw0}
\dpm_{w_0}=\dpm_{w_0\lan
n\ran}=\dpm_1\dpm_2\cdots\dpm_{n-1}\dpm_{w_0\lan n-1\ran}=\cdots.
\end{align}
For $k<n$, let ${}^\dagger S_k\leq S_n$ be the subgroup of
permutations of $\{n-k+1,\ldots,n\}$ and, for $w\in S_k$, let
${}^\dagger w\in{}^\dagger S_k$ denote the corresponding element.
The following useful fact is proved in \cite{EKL}.
\begin{lem}\label{L:daggeridem}
\cite[Lemma~ 3.2]{EKL} The following holds in $\NHm_n$:
$$
\del^-_{w_0}=\del^-_1\del^-_2\cdots\del^-_{n-1}\del^-_{{}^\dagger
w_0\lan n-1\ran}.
$$
\end{lem}
Set
\begin{align}\label{E:ydt}
y^{\dt_n}=(-1)^{{n-1}\choose 2}y_{1}^{n-1}y_{n+2}^{n-2}\cdots y_{n-1}.
\end{align}
\begin{lem}\label{L:Idempotent}
\cite[Proposition~ 3.5]{EKL} \cite[p. 2688]{KL2} We have
\begin{align}
\epm_n=(\pm1)^{n\choose 3}\dpm_{w_0}y^{\dt_n}\in\NH_n^\pm,
\end{align}
\end{lem}

Note that we have adapted this idempotent for \emph{left} modules.

\begin{lem}\label{L:idempotent identities}
\cite[Lemma~ 5]{KL2}\cite[Proposition~ 3.6]{EKL}
The following identity holds in $\NH_n^\pm$:
 $\epm_n\dpm_{w_0}=\dpm_{w_0}.$
\end{lem}

Identify $\epm_{n-1}$ with the image of $\epm_{n-1}\otimes 1$ under
the natural inclusion
$\NH^\pm_{n-1}\otimes\NH^\pm_1\hookrightarrow\NH^\pm_n$   and let
$^\dagger\epm_{n-1}$ denote the image of  $1\otimes\epm_{n-1}$ under
the inclusion
$\NH^\pm_1\otimes\NH^\pm_{n-1}\hookrightarrow\NH^\pm_n$. The
following identities are \cite[(10)-(11)]{KL2}.
\begin{lem}\label{L:IdemIdentities} The following hold in $\NH^\pm_n$:
\begin{align}
\label{E:idempotentmult}&\epm_{n-1}\epm_n=\epm_n,\;\;\;
{}^\dagger\epm_{n-1}\epm_n=\epm_n,\\
\label{E:idempotentcrossing}
 &\epm_n\dpm_1\cdots\dpm_{n-1}\epm_{n-1}
 =\dpm_1\cdots\dpm_{n-1}\epm_{n-1}     ,\\
\label{E:idempotentcrossing2}
 &\epm_n\dpm_1\cdots\dpm_{n-1}{}^\dagger\epm_{n-1}
 =\dpm_1\cdots\dpm_{n-1}{}^\dagger\epm_{n-1}.
\end{align}
\end{lem}

\begin{proof}
Using the definition \eqref{E:Demazure}, formulae
\eqref{E:idempotentmult} reduce to standard properties of the
Demazure operators. Formula \eqref{E:idempotentcrossing} is
immediate from Lemmas \ref{L:Idempotent} and \ref{L:idempotent
identities} since
\begin{align*}
\epm_n\dpm_1\cdots\dpm_{n-1}\epm_{n-1}
&=\epm_n\dpm_1\cdots\dpm_{n-1}\dpm_{w_0\lan n-1\ran}y^{\dt_{n-1}}\\
    &=\epm_n\dpm_{w_0\lan n\ran}y^{\dt_{n-1}}
    =\dpm_{w_0\lan n\ran}y^{\dt_{n-1}}
    =\dpm_1\cdots\dpm_{n-1}\epm_{n-1}.
\end{align*}
The proof of \eqref{E:idempotentcrossing2} is the same, except
inserting Lemma \ref{L:daggeridem} for $\dpm_{w_0}$.
\end{proof}

\begin{lem}\label{L:Annihilators}\cite[(12)-(13)]{KL2}
The following hold in $\NH^\pm_n$:
$$
\dpm_{n-1}\cdots\dpm_2\dpm_1y_1^a\epm_n
=\begin{cases}(-1)^{n-1}\epm_n&\mbox{if }a=n-1\\
 0&\mbox{if }a<n-1,
 \end{cases}
 $$
and
$$
\dpm_1\dpm_2\cdots\dpm_{n-1}y_n^a\epm_n=
\begin{cases}\epm_n&\mbox{if }a=n-1\\
  0&\mbox{if }a<n-1.
\end{cases}
$$
\end{lem}

\begin{proof}
The proof follows by induction using \eqref{E:Relations8} and the
fact that $\dpm_r\epm_n=0$.\end{proof}

\subsection{Categorical Serre relations}\label{subsec:CategoricalSerreRelations}

From now on, we write $\mathbf{e}_{i,n}=\mathbf{e}_n^\pm$ since the
$\pm$ notation can be recovered from the parity of $i$, and
translate $\NH^\pm_n$ to the notation of $\Hm(n\af_i)$. This defines
the unique projective indecomposable $\Hm(n\af_i)$-module
\begin{align}\label{E:Projective}
P_{i^{(n)}} :=\Pi^{p(i){n\choose2}}\Hm(n\af_i)
\mathbf{e}_{i,n}\left\{-{n\choose2}\right\},
\end{align}
and
$$
\Hm(n\af_i)\cong\bigoplus_{[n]_i!}P_{i^{(n)}}.
$$
Here for $f(q,\pi)=\sum_{k\in\Z}(a_{k,\zero}+\pi
a_{k,\one})q^k\in\Z_{\geq0}[q,q^{-1}]^\pi$ and for a
$(\Z\times\Z_2)$-graded vector space $M$, we have set
$$
\bigoplus_{f(q,\pi)}M
=M^{\oplus_f}=\bigoplus_{k\in\Z}\left(M^{\oplus_{a_{k,\zero}}}\oplus(\Pi
M)^{\oplus_{a_{k,\one}}}\right)\{k\}.
$$
For example, if $i\in I_\one$,  $ H(2\af_i)\cong
P_{i^{(2)}}^{\oplus[2]_i!}=(\Pi P_{i^{(2)}})\{1\}\oplus
P_{i^{(2)}}\{-1\}.
$

Now, let $\nu\in Q^+$ be arbitrary. For $\ui\in I^\nu$, we may
consider any grouping of the form
\begin{align}\label{E:uiGrouping}
\ui=(\underbrace{i_1,\ldots,i_1}_{k_1},
\ldots,\underbrace{i_t,\ldots,i_t}_{k_t}).
\end{align}
Identify $\Hm(k_1\af_{i_1})\otimes\cdots\otimes\Hm(k_t\af_{i_t})$
with its image in $\Hm(\nu)$ under the natural embedding, and define
\begin{align*}
\mathbf{e}_{\ui,\uk} := &\mathbf{e}_{i_1,k_1}\otimes\cdots\otimes
\mathbf{e}_{i_t,k_t},\\
 p(\ui,\uk):=& p(i_1){k_1\choose 2}+\ldots+p(i_t){k_t\choose2}, \\
{\uk\choose 2} :=& \sum_{a=1}^t{k_a\choose 2}.
\end{align*}
We further define the projective module
\begin{align}\label{E:Projective module}
P_{\ui^{(\uk)}}:=P_{i_1^{(k_1)}\cdots
i_t^{(k_t)}}=\Pi^{p(\ui,\uk)}\Hm(\nu)
\mathbf{e}_{\ui,\uk}\left\{-{\uk\choose2}\right\}.
\end{align}

Recall $p(k;i,j)$ from \eqref{E:pkij}.

\begin{thm}\label{T:CategoricalSerre}
For $i,j\in I$, $i\neq j$, let $N=1-a_{ij}$. There exists a split
exact sequence of $\Hm(N\alpha_i+\alpha_j)$-modules:
\begin{align*}
\xymatrix{0\ar[r]&\Pi^{p(0;i,j)}P_{i^{(N)}j}\ar[r]
&\Pi^{p(1;i,j)}P_{i^{(N-1)}ji}\ar[r]
&\cdots\ar[r]&\Pi^{p(N;i,j)}P_{ij^{(N)}}\ar[r]&0}.
\end{align*}
In particular, there is an isomorphism
\begin{align*}
\bigoplus_{k=0}^{\lfloor
\frac{N+1}{2}\rfloor}\Pi^{p(2k;i,j)}P_{i^{(N-2k)}ji^{(2k)}}
\cong\bigoplus_{k=0}^{\lfloor
\frac{N}{2}\rfloor}\Pi^{p(2k+1;i,j)}P_{i^{(N-2k-1)}ji^{(2k+1)}}.
\end{align*}
\end{thm}

\begin{proof}
Assume $i$ is even, so $p(k;i,j)\equiv0$. Then, this result is
proved in \cite{KL2} (all the relevant maps being even). The main
technical tools needed are Lemma \ref{L:idempotent identities} and
Lemma \ref{L:Annihilators}. From now on, we assume $i$ is odd. The
necessary technical facts carry over and, once we keep careful track
of the parity, essentially the same proof as in \cite{KL2} works.

In the special case $a_{ij}=0$, the theorem states that there is an
isomorphism $P_{ji}\cong \Pi^{p(j)}P_{ij}$. The relevant map is
given by right multiplication by $\tau_1e(ij)\in\Hm(\af_i+\af_j)$.
In order to be a morphism in $\Projm(\af_i+\af_j)$, this map must
preserve the $\Z_2$-grading. To see this, note that $\tau_1e(ij)$ is
odd when $j$ is odd, and even when $j$ is even. In particular, the
corresponding morphism is a parity preserving isomorphism $P_{ji}\to
\Pi^{p(j)} P_{ij}$, as required.

Now consider the case $a_{ij}\neq0$. Let $n=N+1$, and write
$P_k=P_{i^{(k)}ji^{(n-k-1)}}$, and
$\bfe(k)=\bfe_{i^{(k)}ji^{(n-k-1)}}$ ($0\leq k<n$). Let
$$\af_{k,k+1}=\tau_{n-1}\cdots\tau_{k+2}\tau_{k+1}\bfe({k+1}),$$
and
$$\af_{k+1,k}=\tau_1\tau_2\cdots\tau_{k+1}\bfe(k),$$
for $0\leq k<n$. Right multiplication by $\af_{k,k+1}$ and
$\af_{k+1,k}$ define elements of
$$\Homm_{N\af_i+\af_j}(\Pi^{p(N-k;i,j)}P_{k}, \Pi^{p(N-k-1;i,j)}P_{k+1}),$$
and
$$\Homm_{N\af_i+\af_j}(\Pi^{p(N-k-1;i,j)}P_{k+1}, \Pi^{p(N-k;i,j)}P_{k}),$$
respectively. 

Assume $0\leq k<n-1$. Via the graphical calculus developed in \cite{KL2} one readily shows
by Lemma~ \ref{L:idempotent identities} that
\begin{align*}
\af_{k,k+1}\af_{k+1,k}&=\tau_{n-1}\cdots\tau_{k+1}\bfe(k+1)\tau_1\cdots\tau_{k+1}\bfe(k)\\
    &=\tau_{n-1}\cdots\tau_{k+1}\tau_1\cdots\tau_{k+1}\bfe(k)\\
 &=(-1)^{k-1}(\tau_{n-1}\cdots\tau_{k+2})
  (\tau_1\cdots\tau_{k-1})(\tau_{k+1}\tau_k\tau_{k+1})\bfe(k)\\
 &=(-1)^{(k-1)(n-k-1)}(\tau_1\cdots\tau_{k-1})
 (\tau_{n-1}\cdots\tau_{k+2})(\tau_{k+1}\tau_k\tau_{k+1})\bfe(k)
\end{align*}
where the minus signs are due to \eqref{E:Relations7}. Similarly, for $0<k\leq n-1$,
\begin{align*}
\af_{k,k-1}\af_{k-1,k}&=\tau_1\cdots\tau_k\tau_{n-1}\cdots\tau_k\bfe(k)\\
    &=(-1)^{(n-1)-(k+2)+1}(\tau_1\cdots\tau_{k-1})
    (\tau_{n-1}\cdots\tau_{k+2})(\tau_k\tau_{k+1}\tau_k)\bfe(k),
\end{align*}
cf. \cite[p.2693]{KL2}. Using (C4), $n-1=N=1-a_{ij}$ is odd, so
$(n-1)-(k+2)+1\equiv k$ and $(k-1)(n-k-1)\equiv (k-1)$ (mod 2). Now,
we have
\begin{align}\label{E:alphadifference}
(-1)^{k}(\af_{k,k-1}&\af_{k-1,k}+\af_{k,k+1}\af_{k+1,k})=\\
\nonumber=&(-1)^{k}\af_{k,k-1}\af_{k-1,k}-(-1)^{k-1}\af_{k,k+1}\af_{k+1,k}\\
\nonumber=&\tau_1\cdots\tau_{k-1}\tau_{n-1}\cdots\tau_{k+2}
(\tau_k\tau_{k+1}\tau_k-\tau_{k+1}\tau_k\tau_{k+1})\bfe(k)\\
\nonumber=&\tau_1\cdots\tau_{k-1}\tau_{n-1}\cdots\tau_{k+2}\\
\nonumber&\times\left((-1)^{p(j)}(y_{k+2}-y_k)\frac{\calQ_{ij}
(y_{k+2},y_{k+1})-\calQ_{ij}(y_k,y_{k+1})}{y_{k+2}^2-y_k^2}\right)\bfe(k).
\end{align}
Set $\xi=(-1)^{1+p(j)}$. Then, it follows from the proof of Lemma
\ref{L:Qconditions} that
$$
\calQ_{ij}(u,v)=(-1)^d(u^{2a}+\xi
v^{2b})^c=(-1)^d\sum_{m=0}^c\xi^{c-m}{c\choose
m}u^{2am}v^{2b(c-m)}
$$
for some $a,b,c,d\in\Z_{\geq0}$, where $d=d_{ij}$, $2ac=-a_{ij}$ and
$2bc=-a_{ji}$. Now, a calculation gives
\begin{align}\label{E:Qbraid}
(-1)^{p(j)+d} & (y_{k+2}-y_k)\frac{\calQ_{ij}(y_{k+2},y_{k+1})
-\calQ_{ij}(y_k,y_{k+1})}{y_{k+2}^2-y_k^2} \\
\nonumber&=(-1)^{d+1}\sum_{m=0}^c\sum_{l=1}^{am}\xi^{c-m+1}{c\choose
m}(y_{k+2}^{2(am-l)+1}y_k^{2l-2}-y_{k+2}^{2(am-l)}y_k^{2l-1})y_{k+1}^{2b(c-m)}.
\end{align}
Since $\bfe(k)=\bfe_{i,k}\otimes 1_{\af_j}\otimes \bfe_{i,n-k-1}$,
$y_{k+1}\bfe(k)=\bfe(k)y_{k+1}\bfe(k)$, and we see from Lemma
\ref{L:Annihilators} that the only terms in  terms in
\eqref{E:Qbraid} which are nonzero in \eqref{E:alphadifference}
involve the monomial $y_{k+2}^{n-k-2}y_k^{k-1}$ (see
\cite[p.2694]{KL2} for a graphical interpretation). By degree
considerations, the  terms in \eqref{E:Qbraid} with $m<c$ do not
contribute to \eqref{E:alphadifference}, and so
\begin{align*}
(-1)^{k} &(\af_{k,k-1}\af_{k-1,k}+\af_{k,k+1}\af_{k+1,k})
 \\
=&(-1)^{d+1}\xi\tau_1\cdots\tau_{k-1}\tau_{n-1}\cdots\tau_{k+2}
\sum_{l=1}^{-a_{ij}/2}(y_{k+2}^{-a_{ij}-2l+1}y_k^{2l-2}
-y_{k+2}^{-a_{ij}-2l}y_k^{2l-1})\bfe(k).
\end{align*}
The monomial $y_{k+2}^{n-k-2}y_k^{k-1}$ occurs when
$l=\frac{k+1}{2}$ or $l=\frac{k}{2}$, depending on whether $k$ is
even or odd. In either case, we arrive at
\begin{align}\label{E:afaf+afaf}
(-1)^{k}(\af_{k,k-1}\af_{k-1,k}+\af_{k,k+1}\af_{k+1,k})
=(-1)^{d+1}\xi(-1)^{n-1}\bfe(k)=(-1)^d\xi\bfe(k),
\end{align}
where we have used that $n-1=N=1-a_{ij}$ is odd by (C4).

For $k=n-1$, we have
\begin{align*}
\af_{n-1,n-2}\af_{n-2,n-1}&=\tau_{1}\cdots\tau_{n-2}\tau_{n-1}^2\bfe(n-1)\\
    &=\tau_{1}\cdots\tau_{n-2}\mathsf{Q}_{ij}(y_{n-1},y_{n})\bfe(n-1)\\
    &=(-1)^d\sum_{m=0}^c\xi^{c-m}{c\choose m}y_{n}^{2b(m-c)}(\tau_{1}
    \cdots\tau_{n-2})y_{n-1}^{2am}\bfe(n-1).
\end{align*}
Using Lemma \ref{L:Annihilators} applied to
$\Hm((n-1)\af_i)\subset\Hm(n\af_i)$, the only nonzero term in the
sum above corresponds to $m=c$ (so the exponent of $y_n$ is
$2ac=-a_{ij}=(n-1)-1$), and gives
$$
\af_{n-1,n-2}\af_{n-2,n-1}=(-1)^{d_{ij}}\bfe(n-1).
$$
By a similar argument, adapted to the copy
${}^\dagger\Hm((n-1)\af_i)\subset\Hm(n\af_i)$ of $\Hm((n-1)\af_i)$
\emph{embedded on the right}, we can show that
$$\af_{0,1}\af_{1,0}=(-1)^{d_{ij}}\bfe(0).$$

The maps $\af_{k+1,k}$ are going to be maps in a chain complex.
Therefore, we prove by induction that, for \emph{any} $N\geq2$,
$\af_{k+1,k}\af_{k,k-1}=0$ in $\Hm(N\af_i+\af_j)$, for $1\leq k\leq
N-1$. We want to emphasize that this particular statement holds
independently of the meaning of $N$ given in the statement of the
theorem (as will be necessary when making the inductive step). To
that end, let
\begin{align}\label{E:uik}
\ui_k=(\underbrace{i,\ldots,i}_{k},j,\underbrace{i,\ldots,i}_{n-k-1}).
\end{align}
Then, we have $\bfe(k)=e(\ui_k)\bfe(k)$.

When $k=1$, we calculate as above that
\begin{align*}
\af_{2,1}\af_{1,0}=\tau_1\tau_2\bfe(1)\tau_1\bfe(0)=\tau_1\tau_2\tau_1\bfe(0).
\end{align*}
Recalling that $\bfe(0)=\bfe(0)=1_{\af_j}\otimes e_{i,n-1}$, we
deduce from properties of the spin nilHecke algebra that
$\tau_2\bfe(0)=0$. Therefore, applying \eqref{E:Braid Relations},
\begin{align*}
\af_{2,1}\af_{1,0}=\tau_1\tau_2\tau_1e(\ui_0)\bfe(0)=\tau_2\tau_1\tau_2\bfe(0)=0
\end{align*}
 proving the base case. For the inductive step,
\begin{align*}
\af_{k+1,k}\af_{k,k-1}&=\tau_1\cdots\tau_{k+1}\bfe(k)\tau_1\cdots\tau_k\bfe(k-1)\\
    &=\tau_{1}\cdots\tau_{k+1}\tau_{1}\cdots\tau_{k}e(\ui_{k-1})\bfe(k-1)\\
    &=\tau_{1}\cdots\tau_{k+1}\tau_{1}e(\ui_k)\tau_2\cdots\tau_{k}\bfe(k-1)\\
    &=(-1)^{p(j)+k-2}(\tau_1\tau_2\tau_1)e(\ui_{k+1})(\tau_3\cdots\tau_{k+1}
    \tau_2\cdots\tau_k)\bfe(k-1)\\
    &=(-1)^{p(j)+k-2}(\tau_2\tau_1\tau_2)(\tau_3\cdots\tau_{k+1}\tau_2\cdots\tau_k)\bfe(k-1),
\end{align*}
where we have again use \eqref{E:Braid Relations}. By Lemma
\ref{L:IdemIdentities}, we have that
$$
\bfe(k-1)={}^\dagger\bfe(k-2)\bfe(k-1),
$$
where ${}^\dagger\bfe(k-2)=1_{\af_i}\otimes e_{i,k-2}\otimes
1_{\af_j}\otimes e_{i,n-k}$. Therefore, we may apply induction using
${}^\dagger\Hm((N-1)\af_i+\af_j)\subset\Hm(N\af_i+\af_j)$ to show
that
$$
\tau_2\cdots\tau_{k+1}\tau_2\cdots\tau_k\bfe(k-1)=0.
$$

Finally, we may define isomorphisms
\begin{align*}
\xymatrix{{\displaystyle\bigoplus_{k=0}^{\frac{n-2}{2}}
\Pi^{p(N-2k-1;i,j)}P_{2k+1}}\ar@/^/[r]^{\af'}
&{\displaystyle\bigoplus_{k=0}^{\frac{n}{2}}
\Pi^{p(N-2k;i,j)}P_{2k}}\ar[l]^{\af''}}
\end{align*}
where $\af', \af''$ are given by the sum of maps
$$
\af'=\sum_{k=0}^{n/2-1}(-\xi)(\af_{2k+1,2k}
+\af_{2k+1,2k+2})+\af_{n-1,n-2}
$$
and
$$
\af''=\sum_{k=0}^{n/2}(-1)^{d_{ij}}(\af_{2k,2k+1}
+\af_{2k,2k-1})+(-1)^{d_{ij}}\af_{n-2,n-1}.
$$

For the last step, we need to check that these are actually
morphisms in our category. That is, we need to show that these are
\emph{even} elements of the respective $\Homm$-spaces. Recall
$\ui_k$ from \eqref{E:uik}. We may rewrite
$$
\af_{k,k+1}=\tau_{n-1}\cdots\tau_{k+1}e(\ui_{k+1})\bfe_{k+1}
=\tau_{n-1}e(\ui_{k})\cdots\tau_{k+2}e(\ui_{k})\tau_{k+1}e(\ui_{k+1})\bfe(k+1),
$$
and
$$
\af_{k+1,k}=\tau_1\cdots\tau_{k+1}e(\ui_{k})\bfe_{k}
=\tau_1e(\ui_{k+1})\cdots \tau_{k}e(\ui_{k+1})\tau_{k+1}e(\ui_{k})\bfe(k).
$$
Therefore, the parity of $\af_{k+1,k}$ is $k+p(j)$ (mod 2) and the
parity of $\af_{k,k+1}$ is $(n-1)-(k+2)+1+p(j)\equiv k+p(j)$ (mod
2), where we have again used that $n-1=N=1-a_{ij}$ is odd.

When $j\in I_\zero$, this means we have the following sequence of
morphisms in our category:
$$
\xymatrix{0\ar[r]&P_N\ar[r]&P_{N-1}\ar[r] &\Pi
P_{N-2}\ar[r]\ar[r]&\cdots\ar[r]&\Pi^{N\choose 2}P_0\ar[r]&0}
$$
where the maps are
$$
\af_{N-k,N-k-1}\in\Hom_{N\af_i+\af_j}(\Pi^{k\choose
2}P_{N-k},\Pi^{{k+1}\choose 2}P_{N-k-1}).
$$
For $j\in I_\one$,
$$
\xymatrix{0\ar[r]&P_N\ar[r]&\Pi P_{N-1}\ar[r]
 &\Pi P_{N-2}\ar[r]&\cdots\ar[r]&\Pi^{{N+1}\choose 2}P_0\ar[r]&0}
$$
is a sequence of morphisms in our category, where the maps are given by
$$
\af_{N-k,N-k-1}\in\Hom_{N\af_i+\af_j}(\Pi^{{k+1}\choose
2}P_{N-k},\Pi^{{k+2}\choose 2}P_{N-k-1}).
$$
Now, the proof follows from the simple computation
$$
p(k;i,j)=\begin{cases}{k\choose2}&\mbox{if }j\in
I_\zero,\\{{k+1}\choose2}&\mbox{if }j\in I_\one.\end{cases}
$$
\end{proof}

\begin{example}
Assume $a_{ij}=-4$ (so $n=6$), where $i$ is odd and $j$ is even,
say.  Then, the maps $\af'$ and $\af''$ are defined by summing the
maps down the left and right columns, respectively, in the following
diagram:
\[
\xymatrix{ P_1 \ar[rr]^{-\xi\af_{1,0}}\ar[ddrr]_{-\xi\af_{1,2}}
 && \Pi P_0 \ar[rr]^{(-1)^{d_{ij}}\af_{0,1}} && P_1 \\
 \oplus && \oplus && \oplus \\
 \Pi P_3 \ar[rr]_{-\xi\af_{3,2}}\ar[ddrr]_{-\xi\af_{3,4}}
 && P_2 \ar[rr]_{(-1)^{d_{ij}}\af_{2,3}}\ar[uurr]_{(-1)^{d_{ij}}\af_{2,1}}
 && \Pi P_3 \\
 \oplus
 && \oplus && \oplus \\
 P_5 \ar[rr]_{\af_{5,4}} && \Pi P_4 \ar[rr]_{(-1)^{d_{ij}}\af_{4,5}}
 \ar[uurr]_{\af_{4,3}}
 && P_5 }
\]
In particular, it follows by \eqref{E:afaf+afaf} that the
composition is
\begin{align*}
\af''\af'
 &=(-1)^{d_{ij}}\xi(\af_{1,0}\af_{0,1}+\af_{1,2}\af_{2,1})
  +(-1)^{d_{ij}}\xi(\af_{3,2}\af_{2,3}+\af_{3,4}\af_{4,3})+(-1)^{d_{ij}}
  \af_{5,4}\af_{4,5}\\
 &=\bfe(1)+\bfe(3)+\bfe(5).
\end{align*}
\end{example}

\section{Categorification of quantum superalgebras} \label{S:Categorification}

\subsection{Induction and restriction}\label{subsec:IndRes}

Let $\mu,\nu\in Q^+$. Assume throughout the section that $\height\,
\mu=m$ and $\height\, \nu=n$, and let $D_{m,n}$ be the set of
minimal left $S_m\times S_n$-coset representatives in $S_{m+n}$.
The natural embedding
$\Hm(\mu)\otimes\Hm(\nu)\hookrightarrow\Hm(\mu+\nu)$ maps
$e(\ui)\otimes e(\uj)$ to $e(\ui\uj)$ for all $\ui\in I^\mu$ and
$\uj\in I^\nu$, and $\ui\uj\in I^{\mu+\nu}$ is obtained by
concatenation. The image of the identity element $1_\mu\otimes
1_\nu$ is the idempotent
$$1_{\mu,\nu}=\sum_{\ui\in I^\mu,\;\uj\in I^\nu}e(\ui\uj).$$
Define the functor
\begin{align}
\label{E:Restriction}
\Res_{\mu,\nu}^{\mu+\nu}:\Modm(\mu+\nu)\longrightarrow
\Modm(\mu)\otimes\Modm(\nu),
\end{align}
by
$$\Res_{\mu,\nu}^{\mu+\nu}M=1_{\mu,\nu}M,$$
and the functor
\begin{align}
\label{E:Induction}
\Ind_{\mu,\nu}^{\mu+\nu}:\Modm(\mu)\otimes\Modm(\nu)\longrightarrow
\Modm(\mu+\nu),
\end{align}
by
$$\Ind_{\mu,\nu}^{\mu+\nu}(M\otimes N)=\Hm(\mu+\nu)
1_{\mu,\nu}\bigotimes_{\Hm(\mu)\otimes\Hm(\nu)}(M\otimes N).$$

\begin{prp}\label{P:free}
The module $1_{\mu,\nu}\Hm(\mu+\nu)$ is a free graded left
$\Hm(\mu)\otimes\Hm(\nu)$ module.
\end{prp}

\begin{proof}
The set  $\{1_{\mu,\nu}\tau_w|w\in D_{m,n}\}$ is a basis for
$1_{\mu,\nu}\Hm(\mu+\nu)$ as a free graded left
$\Hm(\mu)\otimes\Hm(\nu)$-module.
\end{proof}

\begin{cor}\label{C:free}
The functors $\Res_{\mu,\nu}^{\mu+\nu}$ and
$\Ind_{\mu,\nu}^{\mu+\nu}$ take finitely generated projective
modules to finitely generated projective modules.
\end{cor}

We have
$$
\Ind_{\mu,\nu}^{\mu+\nu}( P_\ui\otimes P_\uj) =P_{\ui\uj}.
$$
Passing to direct summands, we deduce that the same holds if we
replace $P_\ui$ and $P_{\uj}$ by $P_{\ui^{(\uk)}}$ and
$P_{\uj^{(\ul)}}$, respectively.

For any $\nu\in Q^+$, define the parity
\begin{align}\label{E:pnu}
p(\nu)=p(i_1)+\cdots+ p(i_n),
\end{align}
where $\ui=(i_1,\ldots,i_n)\in I^\nu$.

\begin{thm} [Mackey Theorem]  \label{T:Mackey}
Let $\mu_1,\mu_2,\nu_1,\nu_2\in Q^+$ be such that
$\mu_1+\nu_1=\mu_2+\nu_2=\af$. Then, the graded
$(\Hm(\mu_1)\otimes\Hm(\nu_1),\Hm(\mu_2)\otimes\Hm(\nu_2))$-bimodule
$$
\Res^{\af}_{\mu_1,\nu_1}\Ind_{\mu_2,\nu_2}^{\af}(\Hm(\mu_2)\otimes\Hm(\nu_2))
$$
has a filtration by graded bimodules isomorphic to
\begin{align*}
\Pi^{p(\ld)p(\nu_1+\ld-\nu_2)}\bigg(&(1_{\mu_1}
\Hm(\af)1_{\mu_1-\ld,\ld}\otimes1_{\nu_1}\Hm(\af)1_{\nu_1+\ld-\nu_2,\nu_2-\ld})\\
\otimes_{H}&(1_{\mu_1-\ld,\mu_2-\ld-\mu_1}\Hm(\af)1_{\mu_2}\otimes
1_{\ld,\nu_2-\ld}\Hm(\af)1_{\nu_2})\bigg)\{-(\ld,\nu_1+\ld-\nu_2)\}
\end{align*}
over all $\ld\in Q^+$ such that all the terms above are in $Q^+$. In
this expression, we have denoted
$$
H=\Hm(\mu_1-\ld)\otimes\Hm(\ld)\otimes\Hm(\nu_1+\ld-\nu_2)\otimes\Hm(\nu_2-\ld).
$$
\end{thm}

\begin{proof}
The proof is exactly the same as \cite[Proposition 2.18]{KL1}, with
the parity corresponding to the parity of the diagram appearing in
the proof.
\end{proof}

The following formula for restriction follows from Theorem
\ref{T:Mackey}.

\begin{prp}\label{P:Restriction}
For $\uk\in I^{\mu+\nu}$, we have
$$
\Res_{\mu,\nu}^{\mu+\nu}P_\uk=\bigoplus_{\substack{w\in
D_{m,n}\\\uk=w^{-1}(\ui\uj)}}\Pi^{p(\tau_we(\ui\uj))}(P_\ui\otimes
P_\uj)\{\deg(\tau_we(\ui\uj))\}.
$$
\end{prp}

\subsection{Grothendieck group as a bialgebra}
The exact functors \eqref{E:Restriction} and \eqref{E:Induction}
give rise to exact functors
\begin{align*}
\Ind=\bigoplus_{\mu,\nu} \Ind_{\mu,\nu}^{\mu+\nu}: \bigoplus_{\mu,\nu\in
Q^+}\Modm(\mu)\otimes\Modm(\nu)\longrightarrow \bigoplus_{\ld\in
Q^+}\Modm(\ld),
\end{align*}
and
\begin{align*}
\Res=\bigoplus_{\substack{\ld,\mu,\nu\\\mu+\nu=\ld}}
\Res_{\mu,\nu}^\ld: \bigoplus_{\ld\in Q^+}\Modm(\ld)\longrightarrow
\bigoplus_{\mu,\nu\in Q^+}\Modm(\mu)\otimes\Modm(\nu).
\end{align*}
We set
$$
[\Projm]=\bigoplus_{\nu\in
Q^+}[\Projm(\nu)]\andeqn[\Repm]=\bigoplus_{\nu\in
Q^+}[\Repm(\nu)].
$$

For $x,y\in[\Projm]$, we will simply write $xy=[\Ind](x, y)$. Define
multiplication on $[\Projm]\otimes[\Projm]$ by
\begin{equation}
(x_1\otimes x_2)(y_1\otimes
y_2)=\pi^{p(\mu)p(\nu)}q^{-(\mu,\nu)}(x_1y_1\otimes x_2y_2),
\end{equation}
for $x_1,x_2,y_1,y_2\in[\Projm]$ such that $x_2\in[\Projm(\mu)]$ and
$y_1\in[\Projm(\nu)]$, and the notation $p(\nu)$ is given in
\eqref{E:pnu}.

\begin{prp} \label{P:ResHom}
$[\Projm]$ and $[\Repm]$ with $[\Ind]$ as multiplication are
associative unital $\A^\pi$-algebras. $[\Projm]$ and $[\Repm]$ with
$[\Res]$  are coassociative counital $\A^\pi$-coalgebras. Together,
$([\Projm], [\Ind], [\Res])$ is a bialgebra.
\end{prp}

\begin{proof}
The first two claims are clear from the properties of induction and
restriction functors. The third one means that $[\Res]$ is an
algebra homomorphism, which follows from Theorem \ref{T:Mackey}.
\end{proof}

Recall from \eqref{E:ProjmBilinForm} the bilinear pairing
$(\cdot,\cdot)$ on $[\Projm(\nu)]$. This extends naturally to a
pairing $(\cdot,\cdot)$ on $[\Projm]$ by letting $[\Projm(\nu)]$ be
orthogonal for different $\nu$.

\begin{prp}\label{P:ProjmBilinForm}
The pairing $(\cdot,\cdot)$ on $[\Projm]$ satisfies $(1,1)=1$, and
\begin{enumerate}
\item[(a)]
 $(P_i,P_j)=\delta_{ij}(1-\pi_iq_i^2)^{-1}$;

\item[(b)]
 $(x,yy')=([\Res](x), y\otimes y')$;

\item[(c)]
$(xx',y)=(x\otimes x',[\Res](y))$;

\item[(d)]
$(x\otimes x', y \otimes y') = (x,y)(x', y'),$ for all $x,x',y,y'
\in[\Projm]$.
\end{enumerate}
\end{prp}

\begin{proof}
This is exactly as in \cite[Proposition 3.3]{KL1}. Indeed, (a) is
calculated as the graded dimension of
$e(i)\Hm(\af_i)e(j)=\delta_{ij}\Hm(\af_i)$. Property (b) is
calculated as in \cite[Proposition 3.3]{KL1} using now
\eqref{E:cattensorform}: For $X\in\Projm(\mu+\nu)$,
$Y\in\Projm(\mu)$ and $Y'\in\Projm(\nu)$,
\begin{align*}
([X],[Y][Y'])&=([X],[\Ind_{\mu,\nu}^{\mu+\nu}Y\otimes Y'])\\
 &=\dim_q^\pi(X^\psi\otimes_{\Hm(\mu+\nu)}\Hm(\mu+\nu)1_{\mu,\nu}
 \otimes_{\Hm(\mu)\otimes\Hm(\nu)}Y\otimes Y')\\
 &=\dim_q^\pi(X^\psi1_{\mu,\nu}\otimes_{\Hm(\mu)\otimes\Hm(\nu)}Y\otimes Y')\\
 &=\dim_q^\pi((1_{\mu,\nu}X)^\psi\otimes_{\Hm(\mu)\otimes\Hm(\nu)}Y\otimes Y')\\
 &=([\Res^{\mu+\nu}_{\mu,\nu}X],[Y\otimes Y'])
 =([\Res^{\mu+\nu}_{\mu,\nu}X],[Y]\otimes [Y']).
\end{align*}
A nearly identical calculation gives (c). Finally, to prove (d), it
is enough to observe that there is an even isomorphism of
$(\Z\times\Z_2)$-graded vector spaces
$$
(P\otimes Q)^\psi\otimes_{\Hm(\mu)\otimes\Hm(\nu)}(P'\otimes Q')
\cong(P^\psi\otimes_{\Hm(\mu)}P')\otimes(Q^\psi\otimes_{\Hm(\nu)}Q')
$$
for any $P,P'\in\Projm(\mu)$ and $Q,Q'\in\Projm(\nu)$.
\end{proof}

\begin{example}
We compute $([P_{ii}],[P_{ii}])$ in two ways. Note that $P_{ii}
=\Hm(2\af_i)$. First, by definition and using \eqref{E:dimqpNH}, we
have
$$
([P_{ii}],[P_{ii}])=\dim_q^\pi\Hm(2\af_i)
=\frac{\pi_iq_i^{-1}[2]_i}{(1-\pi_iq_i^2)^2}.
$$
On the other hand, using Propositions \ref{P:Restriction} and
\ref{P:ProjmBilinForm}(b), we deduce that
\begin{align*}
([P_{ii}],[P_{ii}])
    &=([\Res^{2\af_i}_{\af_i,\af_i}P_{ii}], [P_i]\otimes[P_i])\\
    &=((1+\pi^{p(\tau_1)}q^{\deg(\tau_1)})[P_i]\otimes [P_i],[P_i]\otimes[P_i])\\
    &=(1+\pi_iq_i^{-2})([P_i],[P_i])^2.
\end{align*}
By Proposition \ref{P:ProjmBilinForm}(a), the two answers agree,
since $\pi_iq_i^{-1}[2]_i=1+\pi_iq^{-2}$.
\end{example}

\subsection{The homomorphism $\gm^\pi$}

Recall the algebra $\fap$  from Section~\ref{S:QKM}. The following
is a super analogue of \cite[Proposition 3.4]{KL1}.

\begin{prp}\label{P:InjGmHom}
There exists an injective $\A^\pi$-homomorphism
$\gm^\pi:\fap\to[\Projm]$ defined by
$$\gm^\pi(\te_{\ui}^{(\uk)})=P_{\ui^{(\uk)}},
$$
for each
$\te_{\ui}^{(\uk)}:=\te_{i_1}^{(k_1)}\cdots\te_{i_h}^{k_h}\in\fap$.
The following properties hold under $\gm^\pi$:
\begin{enumerate}
%
\item[(a)]
The comultiplication map
$(\fap)_{\mu+\nu}\to(\fap)_\mu\otimes(\fap)_\nu$ corresponds to the
exact functor $\Res^{\mu+\nu}_{\mu,\nu}$.

\item[(b)]
The bar involution on $\fap$ corresponds to the duality functor
$\#$.

\item[(c)]
The bilinear form $(\cdot,\cdot)$ on $\fap$ corresponds to the
bilinear form $(\cdot,\cdot)$ on $[\Projm]$, i.e.,
$(x,y)=(\gm^\pi(x),\gm^\pi(y))$, for all $x,y \in \fap.$
\end{enumerate}
\end{prp}

\begin{proof}
We start with an $\Q(q)^\pi$-homomorphism $\gm^\pi_\Q$ from the free
algebra $\ffpr$ to $[\Projm]$ which sends each $\theta_i$ to
$P_{\ui}$. By Theorem~\ref{th:serre} on quantum Serre and
Theorems~\ref{T:CategoricalSerre} on categorical Serre, $\gm^\pi_\Q$
descends to a homomorphism $\gm^\pi:\fap\to[\Projm]$ as defined in
the theorem. Property~(c) on the compatibility of bilinear forms
follows from Proposition~\ref{prop:bform} and
Proposition~\ref{P:ProjmBilinForm}. The injectivity of $\gm^\pi$
follows from the non-degeneracy of $(\cdot,\cdot)$ on $\fap$.
Property~(a) follows from Proposition \ref{P:ResHom} and then
checking that the homomorphisms $r$ and $[\Res]$ are compatible on
the generators $\theta_i$ and $[P_i]$ under $\gm^\pi$. The
bar-involution on $[\Projm]$ fixes each $P_i$ and satisfies
\eqref{eq:barq}, and (b) follows.
\end{proof}

Let $[\Projm(\nu)]_{\pi=1}$ (resp. $[\Projm(\nu)]_{\pi=-1}$) be the
quotient  of $[\Projm(\nu)]$ by the relations
$$
[M]=[\Pi M], \qquad \text{(respectively, } [M]=-[\Pi M]),
$$
for every $M\in\Projm(\nu)$. Recall the algebra $\fa$ and $\fam$
from Section~\ref{S:QKM}.

\begin{cor}\label{C:InjGmHom} There exist injective homomorphisms
$\gm^+:\fa\to[\Projm]_{\pi=1}$ and $\gm^-:\fam\to[\Projm]_{\pi=-1}$
satisfying the analogous properties of Proposition \ref{P:InjGmHom}.
\end{cor}

\begin{thm}\label{T:Gm+iso}
The map $\gm^+:\fa\to[\Projm]_{\pi=1}$ is an isomorphism of
bialgebras.
\end{thm}

\begin{proof}
Recall $\fa$ is the usual half the quantum Kac-Moody algebra. The
surjectivity of $\gm^+$ can be established following \cite[Chapter
5]{Kl} exactly as was done in \cite[$\S$3.2]{KL1}.\end{proof}

\subsection{The isomorphisms $\gamma^-$ and $\gamma^\pi$}

The arguments in \cite[Chapter 5]{Kl} are insufficient to derive the
surjectivity of $\gm^-$ directly. In particular, the proof of the
independence of characters \cite[Theorem 5.3.1]{Kl} fails since we
have $\ch_q^- L=0$ if $\ch_q^\pi L=\ch_q^\pi \Pi L$ (see
\eqref{E:FormalChar} and \eqref{E:CharSpecialization}). We now give
an alternative argument based on representation theory of quantum
Kac-Moody superalgebras to show this never happens.

Denote by $\Qch M =\sum_\mu \dim M_\mu e^\mu$ the (formal) character
of an algebra or an module $M=\oplus_{\mu} M_\mu$ which is graded by
$Q^+$ and free over the base ring $\A$ or $\Q(q)$. We have a natural
partial order $\geq$ on the collection of characters: for characters
$g, h$, we have $g \ge h$ if and only if $g -h$ is a nonnegative
integer linear combination of $e^\mu$ for $\mu \in Q^+$.

 \begin{lem}  \label{lem:sameCh}
 We have $\Qch \fa =\Qch \fam$.
 \end{lem}

\begin{proof}
By \cite{Kac}, the super Weyl-Kac character formula for integrable
modules holds and therefore so does the super Weyl-Kac denominator
formula. According to \cite{BKM}, the integrable modules of quantum
Kac-Moody superalgebras have the same characters as for their
classical counterpart at $q\mapsto 1$ (generalizing Lusztig's work
for quantum Kac-Moody \cite{Lu1}). Hence the super Weyl-Kac
denominator formula holds, and so
$$
\Qch \fam = \Big(\sum_{w\in W} \text{sgn} (w) e^{w\rho -\rho}
\Big)^{-1},
$$
where $\rho$ and the Weyl group $W$ for the Kac-Moody superalgebra
associated to the GCM $A$ \cite{Kac}  coincide with the counterparts
for the usual Kac-Moody algebra associated to the GCM $A^+$.  On the
other hand, it is well known that $\Qch \fa$ is given by exactly
the same formula (by a combination of Weyl-Kac denominator formula
and Lusztig's result for quantum Kac-Moody). Hence the lemma
follows.
\end{proof}

Recall a finite-dimensional simple module $S$ of an associative
superalgebra $A$ is of type $\texttt M$ if it remains to be simple
with the $\Z_2$-grading forgotten, and is of type $\texttt Q$
otherwise. Recall the parity-shift functor $\Pi$. It is known (cf.
for example \cite[Chapter~12]{Kl}) that a simple $A$-module $S$ is
of type $\texttt Q$ if and only if there exists an even isomorphism
of $A$-modules: $\Pi S \cong S$.

\begin{prop}  \label{equa4chl}
We have
$$
\Qch\fam =  \Qch [\Projm]_{\pi=-1} = \Qch [\Projm]_{\pi=1}
= \Qch \fa.
$$
\end{prop}

\begin{proof}
Using the Cartan pairing \eqref{E:Cartanpairing}, $[\Projm]_{\pi=1}$
has a basis of projective indecomposable modules $P_L$ labeled by
all simple modules $L$ of the spin quiver Hecke algebra (of both
type $\texttt M$ and type $\texttt Q$). On the other hand, since a
type $\texttt Q$ simple module $L$ always admits an even isomorphism
$\Pi L \cong L$,  therefore $[\Projm]_{\pi=-1}$ has a basis labeled
by the type $\texttt M$ simple modules. Hence, we have $\Qch
[\Projm]_{\pi=1} \ge \Qch [\Projm]_{\pi=-1}$. By Theorem~
\ref{T:Gm+iso}, the map $\gm^+$ is an isomorphism. Combining this
with the injection $\gamma^-$ gives us a sequence of inequalities of
formal characters:
$$
\Qch \fam \leq \Qch [\Projm]_{\pi=-1} \leq \Qch
[\Projm]_{\pi=1} = \Qch \fa.
$$
All inequalities much be equalities by Lemma~\ref{lem:sameCh}.
\end{proof}

The equality $\Qch \fam = \Qch [\Projm]_{\pi=-1}$ above
implies the following.
\begin{cor}  \label{cor:iso-}
The map  $\gamma^-_{\Q(q)}: \fqm \rightarrow
{\Q(q)}\otimes_{\A}[\Projm]_{\pi=-1}$ is an isomorphism of $\Z
\times \Z_2$-graded algebras.
\end{cor}

Arguing as in \cite[Proposition 3.20]{KL1} and using
Corollary~\ref{cor:iso-}, we now deduce the part (a) of the
following theorem. Part (b) then follows from (a),
Proposition~\ref{P:InjGmHom} and Theorem~\ref{T:Gm+iso}.

\begin{thm}
\begin{enumerate}
\item[(a)]
The map $\gm^-:\fam\to[\Projm]_{\pi=-1}$ is an isomorphism of
$\Z\times\Z_2$-graded $\A$-algebras.

\item[(b)]
The map $\gm^\pi:\fap\to[\Projm]$ is an isomorphism of
$\Z\times\Z_2$-graded $\A^\pi$-algebras.
\end{enumerate}
\end{thm}

\subsection{The type $\texttt{M}$ phenomenon}
From the proof of  Proposition~\ref{equa4chl}, the identity
$$
\Qch [\Projm]_{\pi=-1} =\Qch [\Projm]_{\pi=1}
$$
holds, and it implies (and is indeed equivalent to) the following
property of simple modules of a spin quiver Hecke algebra (which was
conjectured in \cite[page 3]{KKT}).

\begin{prp}
Each simple module of a spin quiver Hecke algebra is of type
$\texttt M$.
\end{prp}

\end{document}